\documentclass[11pt,a4paper]{amsart}
\usepackage[latin1]{inputenc}
\usepackage{amsmath}
\usepackage{amsfonts}
\usepackage{amssymb}
\usepackage{amsthm}
\usepackage{anysize}
\usepackage{amscd}
\usepackage{hyperref}
\usepackage[square, comma, numbers, sort&compress]{natbib}
\usepackage{indentfirst}
\marginsize{3.0cm}{3.0cm}{3.0cm}{3.0cm}

\newcommand{\form}[1]{{\langle #1 \rangle }}
\newcommand{\pfister}[1]{{\langle \! \langle #1 \rangle \! \rangle}}

\newtheorem{theorem}{Theorem}[section]
\newtheorem{lemma}[theorem]{Lemma}
\newtheorem{proposition}[theorem]{Proposition}
\newtheorem{corollary}[theorem]{Corollary}
\newtheorem{conjecture}[theorem]{Conjecture}

\theoremstyle{definition}
\newtheorem{definition}[theorem]{Definition}
\newtheorem{example}[theorem]{Example}

\newtheorem{questions}[theorem]{Questions}

\theoremstyle{remark}
\newtheorem{remark}[theorem]{Remark}
\newtheorem{remarks}[theorem]{Remarks}

\numberwithin{equation}{section}
\begin{document}

\title{On the splitting of quasilinear $p$-forms}
\author{Stephen Scully}
\address{School of Mathematical Sciences, University of Nottingham, University Park, Nottingham, NG7 2RD, United Kingdom}
\email{pmxss4@nottingham.ac.uk}


\begin{abstract} We study the splitting behaviour of quasilinear $p$-forms in the spirit of the theory of nondegenerate quadratic forms over fields of characteristic different from 2 using an analogue of M. Knebusch's generic splitting tower. Several new applications to the theory of quasilinear quadratic forms are given. Among them, we can mention an algebraic analogue of A. Vishik's theorem on ``outer excellent connections'' in the motives of quadrics, partial results towards a quasilinear analogue of N. Karpenko's theorem on the possible values of the invariant $i_1$, and a proof of a conjecture of D. Hoffmann on quadratic forms with maximal splitting in the quasilinear case. \end{abstract}

\maketitle

\section{Introduction}

Let $k$ be a field of characteristic different from 2. If $q$ is a nondegenerate quadratic form over $k$, its \emph{splitting pattern} may be defined as the increasing sequence $j_0 < j_1 < ... < j_h$ of Witt indices realised by $q$ over all possible field extensions of $k$. This invariant gives a useful means by which to pre-classify quadratic forms according to what one may term their ``algebraic complexity''. A systematic approach to its study was initiated in the 1970's by M. Knebusch (cf. \cite{Knebusch}), who introduced the \emph{generic splitting tower} of a quadratic form $q$, an explicit tower of fields $k_0 \subset k_1 \subset ... \subset k_h$ which splits $q$ in a universal way. From this construction, one naturally extracts the \emph{higher Witt indices} $i_r(q)$ of $q$, and the splitting pattern of $q$ is recovered via the formulae
\begin{equation} \label{eqknebusch} j_s = i_W(q_{k_s}) = \sum_{r=0}^s i_r(q). \end{equation}
Beginning with these observations, the study of the splitting pattern flourished in the subsequent decades. Later, it was observed that the splitting pattern also carries important information about the geometry of the quadric hypersurfaces and higher orthogonal Grassmannians which are naturally associated to quadratic forms. This led to a rich algebro-geometric approach to the splitting pattern which has proved remarkably successful in recent years.

Two principal directions of research emerge. In the first, the main problem is the determination of all possible splitting patterns of quadratic forms over a general field. This is a problem on which substantial progress has been made in the last two decades. As a highlight of this progress, we can mention the following theorem of N. Karpenko which settled a conjecture of D. Hoffmann.

\begin{theorem}[N. Karpenko, \cite{Karpenko}] \label{Hoffmann'sconjecture} Let $q$ be an anisotropic quadratic form of dimension $>1$ over a field of characteristic different from 2. Then $(i_1(q) - 1)$ is the remainder of $(\mathrm{dim}\;q -1)$ modulo some power of 2. \end{theorem}

For a general field $k$, Karpenko's Theorem \ref{Hoffmann'sconjecture} gives a complete set of restrictions on the possible values of the invariant $i_1$. By the very definition of Knebusch's generic splitting tower, the integer $i_1(q)$ coincides with $i_1(q_{r-1})$ for an appropriate quadratic form $q_{r-1}$ defined over the field $k_{r-1}$. Theorem \ref{Hoffmann'sconjecture} therefore puts substantial restrictions on the possible values of the entire splitting pattern. These restrictions are not exhaustive, however, since there exist nontrivial relations among the higher Witt indices. To illustrate this complexity, we recall the following result of A. Vishik.

\begin{theorem}[A. Vishik, \cite{Vishik2}] \label{outerexcellentconnections} Let $q$ be an anisotropic quadratic form of dimension $>1$ over a field $k$ of characteristic different from 2, and write $\mathrm{dim}\;q = 2^n + m$ for uniquely determined integers $n \geq 0$ and $m \in [1,2^n]$. Let $L$ be a field extension of $k$. If $i_W(q_L) < m$, then $i_W(q_L) \leq m - i_1(q)$. \end{theorem}

\begin{remark} \label{excellentconnectionsremark} Theorem \ref{outerexcellentconnections} is only one of several nontrivial applications of the main result of \cite{Vishik2} to the study of the splitting pattern. We refer to \citep[\S 2]{Vishik2} for further details (cf. also Remark \ref{concludingremarks} (2) below). \end{remark}

\begin{proof} Since the statement of Theorem \ref{outerexcellentconnections} is not explicitly formulated in \cite{Vishik2}, let us recall how it can be deduced from \citep[Theorem 1.3]{Vishik2}. Let $Q$ be the projective quadric $\lbrace q = 0 \rbrace$, and let $M(Q)$ denote the motive of $Q$ in the category of Chow motives over $k$ with integer coefficients. Let $U$ be the unique direct summand $N$ of $M(Q)$ such that the Tate motive $\mathbb{Z}$ is isomorphic to a direct summand of $N_{\overline{k}}$ (the ``upper motive'' of $Q$). By \citep[Theorem 1.3]{Vishik2} (on ``outer excellent connections''), the shifted Tate motive $\mathbb{Z}(m-i_1(q))$ is also isomorphic to a direct summand of $U_{\overline{k}}$. Let $s$ be the unique nonnegative integer satisfying $j_s \leq m - 1 < j_{s+1}$. Since $i_W(q) < m$, we have $i_W(q) \leq j_s$. Proving the theorem therefore amounts to showing that $j_s \leq m - i_1(q)$. Suppose that this is not the case, and let $k_s$ be the $(s+1)^{st}$ field in the generic splitting tower of $q$. Then, by a result of M. Rost (cf. \citep[Proposition 2.1]{Vishik3}), $\mathbb{Z}(m-i_1(q))$ is isomorphic to a direct summand of $U_{k_s}$. But $U(i_1(q) - 1)$ is also isomorphic to a direct summand of $M(Q)$ by \citep[Theorem 4.13]{Vishik3}, and it therefore follows that the Tate motive $\mathbb{Z}(m-1)$ is isomorphic to a direct summand of $M(Q)_{k_s}$. Finally, \citep[Proposition 2.6]{Vishik3} now implies that $j_s = i_W(q_{k_s}) \geq m$, which contradicts our choice of $s$. \end{proof}

Much less is known in the second main direction of research, which concerns translating the information contained in the splitting pattern into concrete algebraic terms. To give an explicit example, let us recall another conjecture of D. Hoffmann. Let $q$ be an anisotropic quadratic form of dimension $>1$ over $k$, and write $\mathrm{dim}\;q = 2^n + m$ for uniquely determined integers $n \geq 0$ and $m \in [1,2^n]$. By a result of Hoffmann (\citep[Corollary 1]{Hoffmann1}), it is known that $i_1(q) \leq m$ (note that this is actually a special case of the more general Theorem \ref{Hoffmann'sconjecture}). If equality holds, then we say that $q$ has \emph{maximal splitting}. The maximal splitting property is exhibited by a particularly important class of forms, the so-called \emph{Pfister neighbours}. The following conjecture was formulated in \cite{Hoffmann1} (cf. also \cite{IzhboldinVishik}).

\begin{conjecture} \label{maxsplittingconjecture} Let $q$ be an anisotropic quadratic form over a field of characteristic different from 2 such that $2^n + 2^{n-2} < \mathrm{dim}\;q \leq 2^{n+1}$ for some positive integer $n \geq 2$. If $q$ has maximal splitting, then $q$ is a Pfister neighbour. \end{conjecture}

This conjecture remains wide open. In fact, it is only known in the cases where either $n \leq 4$, or $n \geq 5$ and $\mathrm{dim}\;q \geq 2^{n+1} - 7$ (cf. \cite{IzhboldinVishik}).\\

Over fields of characteristic 2, several additional complications naturally arise. For example, the development of a complete theory of quadratic forms over such a field requires a systematic treatment of singular forms. Geometrically, this means that one has to deal with non-smooth quadrics, and this further amplifies many of the problems which permeate algebraic geometry in positive characteristic (resolution of singularities, construction of cohomological operations). As a result of these additional complexities, the theory of quadratic forms over fields of characteristic 2 is rather underdeveloped in comparison with its characteristic different from 2 counterpart. In particular, much less is known about the splitting patterns of quadratic forms in this setting, even in the nonsingular case (see \cite{Haution} for example). 

One of the main aims of this article is to present some new results on the splitting behaviour of a special class of quadratic forms over fields of characteristic 2, the so-called \emph{quasilinear quadratic forms}. By definition, a quadratic form over a field $k$ of characteristic 2 is quasilinear if it is of Fermat type (that is, can be written in the form $a_1X_1^2 + .... + a_nX_n^2$ for some $a_i \in k$). The theory of quasilinear quadratic forms may be viewed as a direct analogue of the theory of nondegenerate quadratic forms over fields of characteristic 2, in the sense that both theories represent the ``diagonal part'' of the theory of nondegenerate symmetric bilinear forms in their respective characteristics. In the spirit of the corresponding theory over fields of characteristic different from 2, a detailed study quasilinear quadratic forms  was carried out in a series of papers of D. Hoffmann, A. Laghribi and B. Totaro (see \cite{HoffmannLaghribi1}, \cite{HoffmannLaghribi2} and \cite{Totaro} for example). It was later observed by D. Hoffmann that, in many respects, the theory of quasilinear quadratic forms naturally extends to a theory of \emph{quasilinear $p$-forms}, or Fermat-type forms of degree $p$ over fields of characteristic $p$ (\cite{Hoffmann2}). This point of view was further enhanced in \cite{Scully}, where various problems relating to the birational geometry of the zero loci of quasilinear $p$-forms, or \emph{quasilinear $p$-hypersurfaces}, were studied. While many of the most interesting results in the present article are concerned with the special case of quasilinear quadratic forms, we develop the material within the more general framework of the theory of quasilinear $p$-forms as far as we can. \\

The main object of study in this paper is the so-called \emph{standard splitting pattern} of a quasilinear $p$-form. This invariant was introduced and studied in the papers \cite{Laghribi1}, \cite{Laghribi2}, \cite{HoffmannLaghribi1} and \cite{Hoffmann2}. The standard splitting pattern is defined using a construction analogous to that of Knebusch's generic splitting tower for nondegenerate quadratic forms over fields of characteristic different from 2, but comparatively little is known about its properties. In contrast to the situation for nondegenerate quadratic forms over fields of characteristic 2, this construction is not ``universal'' in the sense suggested above. In particular, the formulae \ref{eqknebusch} no longer hold if $s>1$. Despite this deficiency, we show here that the standard splitting pattern possesses other important properties which are not shared by its characteristic different from 2 counterpart. For instance, we show that the standard splitting pattern exhibits strong ``functorial'' properties with respect to rational maps of quasilinear $p$-hypersurfaces (see Theorem \ref{functorialityofssp} for a precise statement when $p \leq 3$). We also show that the standard splitting pattern naturally decomposes into two basic pieces, one of which is ``essentially trivial''. We then obtain further information on the general structure of the ``nontrivial'' component of this decomposition. In particular, we show that in the case where $p=2$, this part of the sequence is monotone increasing (Theorem \ref{hilltheorem}). After establishing these (and other) properties of the standard splitting pattern, we then demonstrate its usefulness in studying the general splitting behaviour of quasilinear $p$-forms by providing several interesting applications. The most significant of these appear in the special case where $p=2$. We can mention, for example, an analogue of Vishik's Theorem \ref{outerexcellentconnections} for quasilinear quadratic forms (Theorem \ref{excellentconnectionsquasilinear}), as well as partial results towards an analogue of Karpenko's Theorem \ref{Hoffmann'sconjecture} (Theorem \ref{i1bounds}, Remark \ref{concludingremarks} (1)). Moreover, we prove that the analogue of Conjecture \ref{maxsplittingconjecture} for quasilinear quadratic forms is true (Theorem \ref{quadraticmaxsplitting}). It is worth remarking that this is the only case in which this result is known in all dimensions. 

This paper may be viewed as a continuation of \cite{Scully}, and we make considerable use of the methods developed there throughout. \\

Throughout this article, $p$ will be an arbitrary prime integer and $F$ a field of characteristic $p$. $\overline{F}$ will denote a fixed algebraic closure of $F$. By a \emph{scheme}, we mean a scheme of finite type over a field. By a \emph{variety}, we mean an integral scheme. If $X$ is a scheme over a field $k$, and $x$ is a point $X$, then $k(x)$ will denote the residue field of the local ring of $X$ at $x$. If, moreover, $X$ is a variety, we will write $k(X)$ for the function field of $X$. A scheme will be called \emph{complete} if it is proper over the base field. Finally, all morphisms and rational maps of schemes are defined relative to the appropriate base field.\\

\noindent {\bf Acknowledgements.} This research is supported by a doctoral training grant at the University of Nottingham. I would like to thank Detlev Hoffmann and Alexander Vishik for very helpful discussions on the subject of the paper. \\

\section{Quasilinear $p$-forms} 

We now recall some of the basic theory of quasilinear $p$-forms as developed in the article \cite{Hoffmann2} of D. Hoffmann. We only discuss the material which will be needed later, and we refer to Hoffmann's paper for any details which we do not provide here.

\subsection{Basic facts.} Let $V$ be a finite dimensional $F$-vector space, and let $\phi \colon V \rightarrow F$ be a homogeneous form of degree $p = \mathrm{char}\;F$.

\begin{definition} In the above notation, the form $\phi$ is called a \emph{quasilinear $p$-form on $V$} if $\phi(v+w) = \phi(v) + \phi(w)$ for all $(v,w) \in V \times V$. \end{definition}

We will say that $\phi$ is a \emph{quasilinear $p$-form over $F$} (or sometimes simply a \emph{form over $F$}) if $\phi$ is a quasilinear $p$-form on some finite dimensional $F$-vector space, which will in turn be denoted by $V_\phi$. In the special case where $p=2$, a quasilinear $p$-form will simply be called a \emph{quasilinear quadratic form}. The dimension $\mathrm{dim}\;\phi$ of $V_\phi$ over $F$ is called the \emph{dimension of $\phi$}. If $\mathrm{dim}\;\phi>1$, we write $X_\phi$ for the projective scheme $\lbrace \phi = 0 \rbrace \subset \mathbb{P}(V_\phi)$ of dimension $\mathrm{dim}\;\phi - 2$. A scheme of this type will be called a \emph{quasilinear $p$-hypersurface}.

A \emph{morphism} $\psi \rightarrow \phi$ of forms over $F$ is an $F$-linear map $f \colon V_\psi \rightarrow V_\phi$ satisfying $\phi(f(v)) = \psi(v)$ for all $v \in V_\psi$. If $f$ is injective, then we say that $\psi$ is a \emph{subform of $\phi$}, and write $\psi \subset \phi$. If $f$ is bijective, $\psi$ and $\phi$ are \emph{isomorphic} and we write $\psi \simeq \phi$. We will say that two forms $\psi$ and $\phi$ over $F$ are \emph{similar} if $\psi \simeq a \phi$ for some $a \in F^*$ (here $a\phi$ is the form on $V_\phi$ defined by $v \mapsto a\phi(v)$). The \emph{direct sum} $\psi \oplus \phi$ and \emph{tensor product} $\psi \otimes \phi$ of forms $\psi$ and $\phi$ are defined in the obvious way. Given a positive integer $n$, we write $n \cdot \phi$ for the direct sum of $\phi$ with itself $n$ times (note that this is not the same as $n \phi$). Given two forms $\psi$ and $\phi$ over $F$, we will say that $\phi$ is \emph{divisible by $\psi$} if there exists a form $\tau$ over $F$ such that $\phi \simeq \psi \otimes \tau$. If $L$ is a field extension of $F$ and $\phi$ is a form over $F$, we write $\phi_L$ for the form over $L$ obtained by the extension of scalars.

If $\phi$ is a quasilinear $p$-form over $F$, then a vector $v \in V_\phi$ is called \emph{isotropic} if $\phi(v) = 0$. We say that the form $\phi$ is \emph{isotropic} if $V_\phi$ contains a nonzero isotropic vector. If $\phi$ is not the zero form, then this can only happen if $\mathrm{dim}\;\phi>1$. In this case, the isotropy of $\phi$ is equivalent to the existence of a rational point on the scheme $X_\phi$. If $V_\phi$ does not contain a nonzero isotropic vector, then $\phi$ is called \emph{anisotropic}. By the definition of a quasilinear $p$-form, the subset of all isotropic vectors in $V_\phi$ is an $F$-linear subspace of $V_\phi$. Its dimension is denoted by $i_0(\phi)$, and is called the \emph{defect index of $\phi$}. In the same way, the set $D(\phi) = \lbrace \phi(v)\;|\;v \in V_\phi \rbrace$ of all values represented by $\phi$ is an $F^p$-linear subspace of $F$. We have the following basic observation.

\begin{lemma}[{cf. \citep[Proposition 2.6]{Hoffmann2}}] \label{valuesdimension} Let $\phi$ be a quasilinear $p$-form over $F$. Then $\mathrm{dim}_{F^p}D(\phi) = \mathrm{dim}\;\phi - i_0(\phi)$. In particular, $\phi$ is anisotropic if and only if $\mathrm{dim}_{F^p}D(\phi) = \mathrm{dim}\;\phi$. 
\begin{proof} Let $U \subset V_\phi$ be a subspace complementary to the subspace of all isotropic vectors in $V_\phi$. We may regard the abelian group $D(\phi)$ as an $F$-vector space, with $a \in F$ acting via left multiplication by $a^p$. Then the evaluation map $\phi \colon U \rightarrow D(\phi)$ is an $F$-linear isomorphism, and since the $F$-dimension of $D(\phi)$ agrees with its $F^p$-dimension, the lemma follows. \end{proof} \end{lemma}

Now, given elements $a_1,...,a_n \in F$, we will write $\form{a_1,...,a_n}$ for the quasilinear $p$-form $a_1X_1^p + ... + a_nX_n^p$ on the $F$-vector space $\bigoplus_{i=1}^m F$ in its standard basis. It is clear from the definition that any quasilinear $p$-form is isomorphic to a form of this type.

\begin{lemma} \label{constructionpforms} Let $U$ be a finite dimensional $F^p$-linear subspace of $F$. Then there exists a unique (up to isomorphism) anisotropic quasilinear $p$-form $\phi$ over $F$ such that $D(\phi) = U$.
\begin{proof} Let $a_1,...,a_n \in F$ be a basis of $U$ over $F^p$, and let $\phi = \form{a_1,...,a_n}$. Then $\phi$ is anisotropic by Lemma \ref{valuesdimension}, and since $D(\phi) = U$, existence is proved. For uniqueness, suppose that $\psi$ is another anisotropic form over $F$ with $D(\psi) = U$. By Lemma \ref{valuesdimension}, there are bases $v_1,...,v_n$ and $w_1,...,w_n$ of $V_\psi$ and $V_\phi$ respectively such that $\psi(v_i) = a_i = \phi(w_i)$ for all $i$. The $F$-linear map $V_\psi \rightarrow V_\phi$ which sends $v_i$ to $w_i$ defines an isomorphism $\psi \simeq \phi$. \end{proof} \end{lemma}

As a corollary, we see that anisotropic forms are determined up to isomorphism by the values they represent.

\begin{corollary}[cf. {\citep[Proposition 2.6]{Hoffmann2}}] \label{anisotropicclassification} Let $\psi$ and $\phi$ be anisotropic quasilinear $p$-forms over $F$. Then $\psi \subset \phi$ if and only if $D(\psi) \subset D(\phi)$. In particular, $\psi \simeq \phi$ if and only if $D(\psi) = D(\phi)$.\end{corollary}

In view of Lemma \ref{constructionpforms}, one can now make the following definition.

\begin{definition} Let $\phi$ be a quasilinear $p$-form over $F$. The unique (up to isomorphism) anisotropic form $\phi_{an}$ over $F$ such that $D(\phi_{an}) = D(\phi)$ is called the \emph{anisotropic part} of $\phi$. \end{definition}

The following statement follows immediately from the proof of Lemma \ref{valuesdimension}.

\begin{proposition}[{cf. \citep[Proposition 2.6]{Hoffmann2}}] \label{classification} Let $\phi$ be a quasilinear $p$-form over $F$. Then $\phi \simeq \phi_{an} \oplus (i_0(\phi) \cdot \form{0})$. \end{proposition}

In summary, we see that the isomorphism class of a quasilinear $p$-form $\phi$ over $F$ is determined by two invariants, the $F^p$-vector space $D(\phi)$ and the defect index $i_0(\phi)$. Clearly we have $\mathrm{dim}\;\phi_{an} \geq 1$. If $\mathrm{dim}\;\phi_{an} = 1$, then we say that $\phi$ is \emph{completely split}. Given two forms $\psi$ and $\phi$ over $F$, we write $\psi \sim \phi$ whenever $\psi_{an} \simeq \phi_{an}$. We conclude this subsection with the following observation, which follows immediately from Lemma \ref{valuesdimension}.

\begin{lemma} \label{indexdirectsum} Let $\psi$ and $\phi$ be anisotropic quasilinear $p$-forms over $F$. Then $i_0(\psi \oplus \phi) = \mathrm{dim}_{F^p}(D(\psi) \cap D(\phi))$. \end{lemma}

\subsection{Quasi-Pfister forms.} In the theory of nondegenerate quadratic forms over fields of characteristic different from 2, an important role is played by the class of so-called \emph{Pfister forms}. In the current setting, one may define analogues of these forms in the following way. First, for any $a \in F$, we write $\pfister{a}$ for the form $\form{1,a,a^2,...,a^{p-1}}$ of dimension $p$ over $F$. Then, given $n$ elements $a_1,...,a_n \in F$, we define a form $\pfister{a_1,...,a_n}$ of dimension $p^n$ over $F$ as the $n$-fold tensor product $\pfister{a_1} \otimes ... \otimes \pfister{a_n}$.

\begin{definition} Let $\pi$ be a quasilinear $p$-form over $F$. Then $\pi$ is called a \emph{quasi-Pfister form} if $\pi = \form{1}$ or $\pi = \pfister{a_1,...,a_n}$ for some $a_i \in F$. \end{definition}

Quasi-Pfister forms were studied extensively in the article \cite{Hoffmann2}, where it was shown that these forms are distinguished by properties completely analogous to those which characterise nondegenerate quadratic Pfister forms over fields of characteristic different from 2. For the moment, we will only need a few simple observations.

\begin{lemma}[{cf. \citep[\S 4]{Hoffmann2}}] \label{Pfisterisotropy} Let $\pi = \pfister{a_1,...,a_n}$ for some $a_i \in F$. Then
\begin{enumerate} 
\item[$\mathrm{(1)}$] $D(\pi) = F^p(a_1,...,a_n)$. In particular, $D(\pi)$ is a field.
\item[$\mathrm{(2)}$] $\pi$ is anisotropic if and only if $[F^p(a_1,...,a_n):F^p] = p^n$.
\item[$\mathrm{(3)}$] $\pi_{an}$ is a quasi-Pfister form. \end{enumerate}
\begin{proof} Part (1) follows from the definition. Given (1), (2) follows from Lemma \ref{valuesdimension}. For (3), we may assume that $\pi$ is not completely split. Now, since $D(\pi)$ is a purely inseparable field extension of $F^p$, there is $m \in [1,n]$ such that $[D(\pi):F^p] = p^m$. After reordering the $a_i$ if necessary, we may assume that $D(\pi) = F^p(a_1,...,a_m)$. By (1), the quasi-Pfister form $\tau = \pfister{a_1,...,a_m}$ over $F$ is anisotropic, and since $D(\tau) = D(\pi)$, we have $\pi_{an} \simeq \tau$.\end{proof} \end{lemma}

\subsection{The norm form and norm degree.} To any quasilinear $p$-form $\phi$, we can associate in a natural way an anisotropic quasi-Pfister form.

\begin{lemma} \label{normformlemma} Let $\phi$ be a quasilinear $p$-form over $F$. Then there exists a unique (up to isomorphism) anisotropic quasi-Pfister form $\phi_{qp}$ over $F$ with the following properties.
\begin{enumerate} \item[$\mathrm{(1)}$] $\phi_{an}$ is similar to a subform of $\phi_{qp}$.
\item[$\mathrm{(2)}$] If $\phi_{an}$ is similar to another anisotropic quasi-Pfister form $\pi$, then $\phi_{qp} \subset \pi$. \end{enumerate}
\begin{proof} The uniqueness is clear. To see the existence, we may assume that $\phi$ is anisotropic. Let $a_1,...,a_n \in F$ be such that $\phi \simeq \form{a_1,...,a_n}$. We may assume that $n > 1$ and $a_1 \neq 0$. Let $\tau = \pfister{\frac{a_2}{a_1},...,\frac{a_n}{a_1}}$, and put $\phi_{qp} = \tau_{an}$. Since $D(\phi) \subset D(a_1\tau) = D(a_1\phi_{qp})$, $\phi$ is similar to a subform of $\phi_{qp}$ by Corollary \ref{anisotropicclassification}. Suppose now that $\phi$ is similar to a subform of another anisotropic quasi-Pfister form $\pi$. Then $aD(\phi) = D(a\phi) \subset D(\pi)$ for some $a \in F^*$. Since $D(\pi)$ is a field, it follows that
\begin{equation*} D(\phi_{qp}) = F^p(\frac{a_2}{a_1},...,\frac{a_n}{a_1}) = F^p(\frac{aa_2}{aa_1},...,\frac{aa_n}{aa_1}) \subset D(\pi). \end{equation*}
By Corollary \ref{anisotropicclassification}, $\phi_{qp} \subset \pi$, as we wanted. \end{proof} \end{lemma}

We can now make the following definitions.

\begin{definition}[{cf. \citep[\S 4]{Hoffmann2}}] The form $\phi_{qp}$ of Lemma \ref{normformlemma} is called the \emph{norm form of $\phi$}. The dimension of $\phi_{qp}$ is called the \emph{norm degree of $\phi$}, and is denoted by $\mathrm{ndeg}\;\phi$. \end{definition}

\begin{remark} \label{completelysplitndeg} By its definition, the norm degree is always a power of $p$. By Lemma \ref{normformlemma} (1), we have $\mathrm{ndeg}\;\phi = 1$ if and only if $\phi$ is completely split. \end{remark}

\section{Quasilinear $p$-forms and extensions of the base field}

We now collect some basic facts concerning the behaviour of quasilinear $p$-forms over extensions of the base field. Again, we remark that most (but not all) of the material in this section can be found in the article \cite{Hoffmann2}.

\subsection{Some general observations.} We begin by noting the following lemma, which will be used repeatedly in what follows.

\begin{lemma}[{\citep[Lemma 5.1]{Hoffmann2}}] \label{excellence} Let $\phi$ be a quasilinear $p$-form over $F$, and let $L$ be a field extension of $F$. Then there exists a subform $\psi \subset \phi$ such that $(\phi_L)_{an} \simeq \psi_L$.
\begin{proof} The point to observe is that the $L^p$-vector space $D(\phi_L)$ is spanned by elements of $D(\phi)$. In particular, we can find a basis $a_1,...,a_n$ of $D(\phi_L)$ consisting of elements of $D(\phi)$. By Lemma \ref{constructionpforms} and Corollary \ref{anisotropicclassification}, the form $\psi = \form{a_1,...,a_n}$ has required property. \end{proof} \end{lemma}

Next we consider the behaviour of the norm degree invariant under field extensions. The following lemma is clear from the construction of the norm form (cf. Lemma \ref{normformlemma} and its proof).

\begin{lemma}[{cf. \citep[Remark 4.11]{Hoffmann2}}] \label{normformext} Let $\phi$ be a quasilinear $p$-form over $F$, and let $L$ be a field extension of $F$. Then $(\phi_L)_{qp}$ is the anisotropic part of $(\phi_{qp})_L$. In particular, $\mathrm{ndeg}\;\phi_L < \mathrm{ndeg}\;\phi$ if and only if $(\phi_{qp})_L$ is isotropic. \end{lemma}

As a corollary, we get the following useful statement.

\begin{corollary}[{\citep[Proposition 5.2]{Hoffmann2}}] \label{normdegcriterion} Let $\phi$ be an anisotropic quasilinear $p$-form over $F$, and let $L$ be a field extension of $F$. If $\phi_L$ is isotropic, then $\mathrm{ndeg}\;\phi_L < \mathrm{ndeg}\;\phi$.
\begin{proof} In view of Lemma \ref{normformext}, we just need to check that $(\phi_{qp})_L$ is isotropic whenever $\phi_L$ is. This is clear, since $\phi$, being anisotropic, is similar to a subform of $\phi_{qp}$ by Lemma \ref{normformlemma}. \end{proof} \end{corollary}

Recall that if $K$ and $L$ are field extensions of $F$, then an $F$-place $K \rightharpoonup L$ is a local $F$-algebra homomorphism $R \rightarrow L$, where $R$ is a valuation subring of $K$ containing $F$. This notion will be used extensively in what follows. We refer to the appendix for some basic facts concerning places and the relevant notation. The following lemma describes the isotropy behaviour of quasilinear $p$-forms in the presence of a place.

\begin{lemma} \label{placeisotropy} Let $K$ and $L$ be field extensions of $F$, and let $\phi$ be a quasilinear $p$-form over $F$. Assume that there exists an $F$-place $K \rightharpoonup L$. Then
\begin{enumerate} \item[$\mathrm{(1)}$] $i_0(\phi_L) \geq i_0(\phi_K)$.
\item[$\mathrm{(2)}$] $\mathrm{ndeg}\;\phi_L \leq \mathrm{ndeg}\;\phi_K$. \end{enumerate}
In particular, if $K \sim_F L$, then $i_0(\phi_K) = i_0(\phi_L)$ and $\mathrm{ndeg}\;\phi_K = \mathrm{ndeg}\;\phi_L$.
\begin{proof} In view of Lemma \ref{normformext}, it will be sufficient to prove (1). We may assume that $\mathrm{dim}\;\phi >1$. Let $m = i_0(\phi_K)$. By Lemma \ref{excellence}, it is enough to show that every codimension $(m-1)$ subform of $\phi$ becomes isotropic over $L$. Let $\psi$ be any such subform. Then the scheme $X_\psi$ has a $K$-valued point, and we want to show that it has an $L$-valued point. Since $X_\psi$ is complete, this follows from Lemma \ref{placescomplete}. \end{proof} \end{lemma}

\subsection{Separable extensions.} Let $k$ be a field, and let $L$ be a field extension of $k$. Recall that the extension $k \subset L$ is called \emph{separable} if the ring $L \otimes_k \overline{k}$ is reduced (that is, has no nonzero nilpotent elements), where $\overline{k}$ is an algebraic closure of $k$. If $k$ has characteristic 0, then every extension of $k$ is separable. In positive characteristic, we have the following well-known result of S. MacLane (cf. \cite{MacLane}, or \citep[Proposition VIII.4.1]{Lang}).

\begin{theorem} \label{MacLane} Let $k$ be a field of characteristic $p>0$, and let $L$ be a field extension of $k$. The following conditions are equivalent.
\begin{enumerate} \item[$\mathrm{(1)}$] The extension $k \subset L$ is separable.
\item[$\mathrm{(2)}$] $L$ is linearly disjoint from $k^{1/p} = k(\sqrt[p]{a}\;|\;a \in k)$ over $k$. \end{enumerate} \end{theorem}

This has the following consequence for quasilinear $p$-forms.

\begin{lemma}[{\citep[Proposition 5.3]{Hoffmann2}}] \label{sepext} Let $\phi$ be a quasilinear $p$-form over $F$, and let $F \subset L$ be a separable extension.
\begin{enumerate} \item[$\mathrm{(1)}$] If $\phi$ is anisotropic, then so is $\phi_L$.
\item[$\mathrm{(2)}$] $\mathrm{ndeg}\;\phi_L = \mathrm{ndeg}\;\phi$. \end{enumerate}
\begin{proof} By Lemma \ref{normformext}, it suffices to prove (1). By Theorem \ref{MacLane}, $L$ is linearly disjoint from $F^{1/p}$ over $F$. It follows that $L^p$ is linearly disjoint from $F$ over $F^p$. In other words, anisotropic quasilinear $p$-forms over $F$ remain anisotropic over $L$. \end{proof} \end{lemma}

\subsection{Purely inseparable extensions of degree $p$.} Any extension of fields may be presented as a separable extension followed by a purely inseparable algebraic extension. In view of Lemma \ref{sepext}, in order to study the isotropy behaviour of quasilinear $p$-forms over extensions of the base field, we are essentially reduced to considering finite purely inseparable extensions. Let $\phi$ be a quasilinear $p$-form over $F$, and let $a \in F \setminus F^p$. Recall that we write $F_a$ for the field $F(\sqrt[p]{a})$. By Lemma \ref{excellence}, there is a subform $\psi \subset \phi$ such that $(\phi_{F_a})_{an} \simeq \psi_{F_a}$. We have the following inclusion and equalities of $F^p$-vector spaces.
\begin{equation*} D(\phi) \subset D(\pfister{a} \otimes \phi) = D(\phi_{F_a}) = D(\psi_{F_a}) = D(\pfister{a} \otimes \psi). \end{equation*}
Moreover, $\pfister{a} \otimes \psi$ is anisotropic by Lemma \ref{valuesdimension} and the choice of $\psi$. Using Corollary \ref{anisotropicclassification}, we get the following result.

\begin{lemma} \label{pinsepisoandmult} In the above notation, $\phi_{an} \subset \pfister{a} \otimes \psi \simeq (\pfister{a} \otimes \phi)_{an}$. \end{lemma}

Now we can prove the following lemma.

\begin{lemma}[{cf. \citep[\S 5]{Hoffmann2}}] \label{pinsepisotropy} Let $\phi$ be a quasilinear $p$-form over $F$, and let $a \in F \setminus F^p$.
\begin{enumerate} \item[$\mathrm{(1)}$] $pi_0(\phi_{F_a}) = i_0(\pfister{a} \otimes \phi)$.
\item[$\mathrm{(2)}$] If $\phi$ is anisotropic, then $\mathrm{dim}(\phi_{F_a})_{an} \geq \frac{1}{p}\mathrm{dim}\;\phi$.
\item[$\mathrm{(3)}$] $\mathrm{ndeg}\;\phi_{F_a} = \begin{cases} \frac{1}{p}\mathrm{ndeg}\;\phi & \text{ if } a \in D(\phi_{qp}) \\
\mathrm{ndeg}\;\phi & \text{ if } a \notin D(\phi_{qp}). \end{cases}$
\item[$\mathrm{(4)}$] If $\phi$ is anisotropic and $\phi_{F_a}$ is isotropic, then $a \in D(\phi_{qp})$. \end{enumerate}
\begin{proof} Parts (1) and (2) follows immediately from Lemma \ref{pinsepisoandmult}. By Lemma \ref{normformext}, $\mathrm{ndeg}\;\phi_{F_a}$ is the dimension of the anisotropic part of $(\phi_{qp})_{F_a}$. By Lemma \ref{pinsepisoandmult}, this is in turn equal to $\frac{1}{p}$ times the dimension of the anisotropic part of $\pfister{a} \otimes \phi_{qp}$. By Lemma \ref{Pfisterisotropy}, the latter integer is equal to $\mathrm{dim}\;\phi_{qp}$ if $a \in D(\phi_{qp})$, and $p\mathrm{dim}\;\phi_{qp}$ otherwise. Since $\mathrm{ndeg}\;\phi = \mathrm{dim}\;\phi_{qp}$ by definition, this proves (3). In view of (3), (4) follows from Corollary \ref{normdegcriterion}. \end{proof} \end{lemma}

The following lemma will also be useful in what follows.

\begin{lemma} \label{pinsepisotropysubform} Let $\phi$ be a quasilinear $p$-form over $F$, let $a \in F \setminus F^p$, and let $m$ be a positive integer. If $i_0(\phi_{F_a}) \geq m$, then $\phi$ contains a subform $\psi$ of dimension $\leq pm$ such that $i_0(\psi_{F_a}) \geq m$.
\begin{proof} We proceed by induction on $m$. Let $w \in V_\phi \otimes_F F_a$ be a nonzero isotropic vector for $\phi_{F_a}$, and write $w$ in the form
\begin{equation*} w = v_0 \otimes 1 + v_1 \otimes \sqrt[p]{a} + ... + v_{p-1} \otimes (\sqrt[p]{a})^{p-1}, \end{equation*}
where the $v_i$ belong to $V_\phi$. Let $U \subset V$ be the subspace generated by the vectors $v_i$, and let $\sigma = \phi|_U$. Clearly $\mathrm{dim}\;\sigma \leq p$, and if $m=1$, we may take $\psi = \sigma$. If $m>1$, let $\tau \subset \phi$ be such that $\phi \simeq \sigma \oplus \tau$. By Lemma \ref{excellence}, there exists a subform $\eta \subset \sigma$ such that $(\sigma_{F_a})_{an} \simeq \eta_{F_a}$. Again, let $\rho \subset \sigma$ be such that $\sigma \simeq \eta \oplus \rho$, and put $\gamma = \eta \oplus \tau$. Then $\phi_{F_a} \sim \gamma_{F_a}$, so that $i_0(\gamma_{F_a}) = i_0(\phi_{F_a}) - \mathrm{dim}\;\rho \geq m - \mathrm{dim}\;\rho$. By the induction hypothesis, there exists a subform $\psi' \subset \gamma$ of dimension $\leq p(m-\mathrm{dim}\;\rho)$ such that $i_0(\psi'_{F_a}) \geq m - \mathrm{dim}\;\rho$. The form $\psi = (\psi' \oplus \sigma)_{an} \subset \phi$ has the properties we want. \end{proof} \end{lemma}

In some special cases, we can say more.

\begin{lemma} \label{weirdlemma} Let $\phi$ be an anisotropic quasilinear $p$-form over $F$, let $a \in F \setminus F^p$, and let $m$ be a positive integer. If $(p^2 - p - 1)i_0(\phi_{F_a}) - (p^2 - 2p)\mathrm{dim}\;\phi \geq m$, then there exists a form $\tau$ of dimension $m$ over $F$ such that $\pfister{a} \otimes \tau \subset \phi$.
\begin{proof} By Lemma \ref{excellence}, there exists a subform $\psi \subset \phi$ such that $(\phi_{F_a})_{an} \simeq \psi_{F_a}$. For all $i \in [1,p-1]$, we define $F^p$-vector spaces $V_i$ by setting
\begin{equation*} V_i = \lbrace b \in D(\psi)\;|\;a^ib \in D(\phi) \rbrace = D(\psi) \cap a^{p-i}D(\phi) \subset D(\psi). \end{equation*}
By Lemma \ref{constructionpforms}, there exists a unique (up to isomorphism) subform $\tau \subset \psi$ satisfying $D(\tau) = \bigcap_{i=1}^{p-1}V_i$. By Lemma \ref{pinsepisotropy} (1), $\pfister{a} \otimes \tau$ is anisotropic, and since $a^i D(\tau) \subset D(\phi)$ for all $i$, we have $\pfister{a} \otimes \tau \subset \phi$ by Corollary \ref{anisotropicclassification}. It now remains to show that $\mathrm{dim}\;\tau \geq n$. By Corollary \ref{indexdirectsum}, we have
\begin{equation*} \mathrm{dim}_{F^p}V_i = \mathrm{dim}_{F^p}(D(\psi) \cap a^{p-i}D(\phi)) = i_0(\psi \oplus a^{p-i} \phi) \end{equation*}
for all $i$. Since each $\psi \oplus a^{p-i}\phi$ is a subform of codimension $(p-2)\mathrm{dim}\;\phi + i_0(\phi_{F_a})$ in $\pfister{a} \otimes \phi$, we therefore have (using Lemma \ref{pinsepisotropy} (1))
\begin{align*} \mathrm{dim}_{F^p}V_i = i_0(\psi \oplus a^{p-i} \phi) & \geq  i_0(\pfister{a} \otimes \phi) - (p-2)\mathrm{dim}\;\phi - i_0(\phi_{F_a})\\
&= pi_0(\phi_{F_a}) - (p-2)\mathrm{dim}\;\phi - i_0(\phi_{F_a})\\
&= (p-1)i_0(\phi_{F_a}) - (p-2)\mathrm{dim}\;\phi. \end{align*}
Finally, we have
\begin{align*} \mathrm{dim}\;\tau = \mathrm{dim}_{F^p}\bigcap_{i=1}^{p-1}V_i & \geq  \sum_{i=1}^{p-1}\mathrm{dim}_{F^p}V_i - (p-2)\mathrm{dim}\;\psi\\
& \geq (p-1)((p-1)i_0(\phi_{F_a}) - (p-2)\mathrm{dim}\;\phi) - (p-2)(\mathrm{dim}\;\phi - i_0(\phi_{F_a}))\\
&= (p^2 - p - 1)i_0(\phi_{F_a}) - (p^2 - 2p)\mathrm{dim}\;\phi. \end{align*}
Since the latter integer is $\geq n$ by assumption, the lemma is proved. \end{proof} \end{lemma}

For the prime 2, we reach the following conclusion.

\begin{corollary}[{\citep[Proposition 7.18]{Hoffmann2}}] \label{isotropyquadratic} Assume that $p=2$. Let $\phi$ be an anisotropic quasilinear quadratic form over $F$, let $a \in F \setminus F^2$, and let $m$ be a positive integer. Then $i_0(\phi_{F_a}) \geq m$ if and only if there exists a form $\tau$ of dimension $m$ over $F$ such that $\pfister{a} \otimes \tau \subset \phi$.
\begin{proof} The implication $\Leftarrow$ is clear, while the converse follows from Lemma \ref{weirdlemma}. \end{proof} \end{corollary}

\begin{remark} The implication $\Rightarrow$ in the above corollary is generally false for $p>2$. \end{remark}

\section{Function fields of quasilinear $p$-hypersurfaces and the standard splitting pattern}

The standard splitting pattern of a quasilinear $p$-form was introduced in the articles \cite{Laghribi1} and \cite{Hoffmann2}. In \S 4.2, we recall its definition and basic properties. We begin with some generalities concerning function fields of quasilinear $p$-hypersurfaces.

\subsection{Function fields of quasilinear $p$-hypersurfaces.} Let $\phi$ be a quasilinear $p$-form of dimension $>1$ over $F$, and let $X_\phi$ be the associated quasilinear $p$-hypersurface. The following lemma is an easy calculation (cf. \citep[Lemma 7.1]{Hoffmann2}).

\begin{lemma} In the above notation, the scheme $X_\phi$ is integral if and only if $\mathrm{ndeg}\;\phi >1$. \end{lemma}

In other words, the scheme $X_\phi$ is a variety provided that $\phi$ is not completely split (cf. Remark \ref{completelysplitndeg}). If $X_\phi$ is a variety, then we will denote its function field by $F(\phi)$. Clearly this field is invariant under multiplying $\phi$ by a scalar. Given quasilinear $p$-forms $\phi_1,...,\phi_n$ of dimension $>1$ over $F$, we will write $F(\phi_1 \times ... \times \phi_n)$ for the function field of the scheme $X_{\phi_1} \times ... \times X_{\phi_n}$ whenever it is integral. Furthermore, we will sometimes simplify the notation where it is appropriate. For example, if $\phi_1 = ... = \phi_n = \phi$, we will simply write $F(\phi^{\times n})$ instead of $F(\phi_1 \times ... \times \phi_n)$. Finally, if $L$ is a field extension of $F$, then we will typically write $L(\phi)$ instead of $L(\phi_L)$ whenever the latter is defined.

\begin{remarks} \label{ffremarks} Let $\phi$ be a quasilinear $p$-form over $F$. Assume that $\phi$ is not completely split.
\begin{enumerate} \item[$\mathrm{(1)}$] The field $F(\phi)$ can be written as a degree $p$ purely inseparable extension of a purely transcendental extension of $F$.
\item[$\mathrm{(2)}$] It follows from Proposition \ref{classification} that the varieties $X_\phi$ and $X_{\phi_{an}}$ are stably birational. In particular, we have $F(\phi) \sim_F F(\phi_{an})$ (cf. Example \ref{stabbireq}).
\item[$\mathrm{(3)}$] If $\psi$ is another quasilinear $p$-form of dimension $>1$ over $F$, then $\psi_{F(\phi)}$ is isotropic if and only if there exists a rational map $X_\phi \dashrightarrow X_\psi$. In particular, $\phi_{F(\phi)}$ is isotropic. \end{enumerate} \end{remarks}

We will be interested in the behaviour of quasilinear $p$-forms over fields of the above kind. In view of Remark \ref{ffremarks} (1), the following result follows from Lemmas \ref{sepext} and \ref{pinsepisotropy}.

\begin{lemma} \label{ffisotropy} Let $\phi$ and $\psi$ be quasilinear $p$-forms over $F$. Assume that $\psi$ is not completely split.
\begin{enumerate} \item[$\mathrm{(1)}$] $\mathrm{ndeg}\;\phi_{F(\psi)} \geq \frac{1}{p} \mathrm{ndeg}\;\phi$.
\item[$\mathrm{(2)}$] If $\phi$ is anisotropic, then $\mathrm{dim}\;(\phi_{F(\psi)})_{an} \geq \frac{1}{p}\mathrm{dim}\;\phi$, and equality holds in (1) if $\phi_{F(\psi)}$ is isotropic. \end{enumerate} \end{lemma}

To facilitate a systematic study of quasilinear $p$-forms over function fields of quasilinear $p$-hypersurfaces, it would be desirable to know whether the relation ``$\phi$ is isotropic over $F(\psi)$'' is transitive. We ask the following general questions.

\begin{questions} \label{questions} Let $\phi$ and $\psi$ be anisotropic quasilinear $p$-forms of dimension $>1$ over $F$, and let $L$ be any field extension of $F$ such that $\psi_L$ is isotropic.
\begin{enumerate} \item[$\mathrm{(1)}$] Does there exist an $F$-place $F(\psi) \rightharpoonup L$?
\item[$\mathrm{(2)}$] If $\phi_{F(\psi)}$ is isotropic, must $\phi_L$ be isotropic also? \end{enumerate} \end{questions}

In view of Lemma \ref{placeisotropy} (1), a positive answer to the first question implies a positive answer to the second. We expect a positive answer to both. The following lemma settles a useful special case.

\begin{lemma}[{cf. \citep[Proposition 7.17]{Hoffmann2}}] \label{subformtransitivity} Let $\phi$, $\psi$ and $\sigma$ be anisotropic quasilinear $p$-forms of dimension $>1$ over $F$. Assume that $\sigma \subset \psi$. Then
\begin{enumerate} \item[$\mathrm{(1)}$] There exists an $F$-place $F(\psi) \rightharpoonup F(\sigma)$.
\item[$\mathrm{(2)}$] If $\phi_{F(\psi)}$ is isotropic, then $\phi_{F(\sigma)}$ is isotropic. \end{enumerate}
\begin{proof} As remarked above, it is enough to prove the first statement. But the canonical closed embedding $X_\sigma \subset X_\psi$ is regular, and so $X_\psi$ is regular at the generic point of $X_\sigma$. The existence of an $F$-place $F(\psi) \rightharpoonup F(\sigma)$ therefore follows from Lemma \ref{placeregpoint}. \end{proof} \end{lemma}

Here is a useful consequence of this observation.

\begin{proposition}[{\citep[Lemma 7.12]{Hoffmann2}}] \label{ndegfunctoriality} Let $\phi$ and $\psi$ be anisotropic quasilinear $p$-forms of dimension $>1$ over $F$. If $\phi_{F(\psi)}$ is isotropic, then $\psi_{qp} \subset \phi_{qp}$. In particular, $\mathrm{ndeg}\;\psi \leq \mathrm{ndeg}\;\phi$.
\begin{proof} By Corollary \ref{anisotropicclassification}, we have to show that $D(\psi_{qp}) \subset D(\phi_{qp})$. Let $a \in D(\psi_{qp})$. If $a \in F^p$, then clearly $a \in D(\phi_{qp})$. Otherwise, the binary form $\tau = \form{1,a}$ is a subform of $\psi_{qp}$. By Lemma \ref{subformtransitivity}, $\phi_{F(\tau)}$ is isotropic. But $F(\tau) = F_a$, so $\phi_{F_a}$ is isotropic. Finally, Lemma \ref{pinsepisotropy} shows that $a \in D(\phi_{qp})$, as we wanted. \end{proof} \end{proposition}

In complete generality, Questions \ref{questions} (1) and (2) remain open whenever $p>3$. However, it was essentially shown in \cite{Scully} that one can settle both problems when $p=2$ or $p=3$ using a sequence of elementary arguments.

\begin{proposition} \label{23transitivity} If $p=2$ or $p=3$, Questions \ref{questions} (1) and (2) have positive answers.
\begin{proof} We remark again that it suffices to prove that Question \ref{questions} (1) has a positive answer. Now, by Lemma \ref{sepext} (1), there is a tower $F \subset M \subset M' \subset L$ of fields such that $M \subset M'$ is purely inseparable of degree $p$, $\psi_M$ is anisotropic, and $\psi_{M'}$ is isotropic. We have trivial $F$-places $F(\psi) \rightharpoonup M(\psi_M)$ and $M' \rightharpoonup L$. Since $F$-places can be composed (cf. the appendix), we reduce to the case where $F \subset L$ is purely inseparable of degree $p$. Under this assumption, Lemma \ref{pinsepisotropysubform} shows that there is a subform $\tau \subset \psi$ of dimension $\leq p$ such that $\tau_L$ is isotropic. By Lemma \ref{subformtransitivity} (1), there is an $F$-place $F(\psi) \rightharpoonup F(\tau)$. Again, since $F$-places can be composed, this further reduces the problem to the case where $\mathrm{dim}\;\psi \leq p$. Now, by Lemma \ref{placeregpoint}, it will suffice to show (under the assumptions $\mathrm{dim}\;\psi \leq p \leq 3$) that $X_\psi$ has a regular $L$-valued point. This was done in \citep[Proposition 4.8]{Scully}. \end{proof} \end{proposition}

Taking Lemma \ref{placeisotropy} (1) into account, we get the following corollary.

\begin{corollary} \label{23minimalityi1} Assume that $p=2$ or $p=3$. Let $\phi$ be an anisotropic $p$-form of dimension $>1$ over $F$, and let $L$ be a field extension of $F$ such that $\phi_L$ is isotropic. Then $i_0(\phi_L) \geq i_0(\phi_{F(\phi)})$. \end{corollary}

\subsection{The standard splitting pattern.} We can now introduce the standard splitting pattern, which is defined via a construction analogous to M. Knebusch's construction of the generic splitting tower of a nondegenerate quadratic form over a field of characteristic different from 2.

\begin{definition}[{cf. \citep[\S 7.5]{Hoffmann2}}] Let $\phi$ be a quasilinear $p$-form over $F$. Assume that $\phi$ is not completely split. Set $F_0 = F$, $\phi_0 = \phi_{an}$ and define inductively
\begin{itemize} \item $F_r = F_{r-1}(\phi_{r-1})$ (provided $\phi_{r-1}$ is not completely split).
\item $\phi_r = ((\phi_{r-1})_{F_r})_{an}$ (provided $F_r$ is defined). \end{itemize}
By Remark \ref{ffremarks} (3), we have $\mathrm{dim}\;\phi_r < \mathrm{dim}\;\phi_{r-1}$. The process is therefore finite, and stops at the first positive integer $h(\phi)$ such that $\phi_{h(\phi)}$ is completely split.
\begin{itemize} \item The integer $h(\phi)$ is called the \emph{height} of $\phi$.
\item The tower $F = F_0 \subset F_1 \subset ... \subset F_{h(\phi)}$ of fields is called the \emph{standard splitting tower of $\phi$}.
\item The integer $i_r(\phi) = i_0((\phi_{r-1})_{F_r})$ is the \emph{$r^{th}$ higher defect index of $\phi$}.
\item The decreasing sequence $\mathrm{sp}(\phi) = (\mathrm{dim}\;\phi_{an} = \mathrm{dim}\;\phi_0, \mathrm{dim}\;\phi_1,...,\mathrm{dim}\;\phi_{h(\phi)} = 1)$ of positive integers is called the \emph{standard splitting pattern of $\phi$} (note that the notation here differs slightly from \cite{Hoffmann2}).
\item We also introduce the sequence $\widetilde{\mathrm{sp}}(\phi) = (\mathrm{dim}\;\phi_1, \mathrm{dim}\;\phi_2,..., \mathrm{dim}\;\phi_{h(\phi)} = 1)$, which is just $\mathrm{sp}(\phi)$ with the first entry removed. \end{itemize} \end{definition}

\begin{remarks} \label{sspremarks} Let $\phi$ be a quasilinear $p$-form over $F$. Assume that $\phi$ is not completely split.
\begin{enumerate} \item[$\mathrm{(1)}$] Since everything is defined inductively, we have $i_r(\phi) = i_1(\phi_{r-1})$ and $\widetilde{\mathrm{sp}}(\phi) = \mathrm{sp}(\phi_1)$.
\item[$\mathrm{(2)}$] Let $F = F_0 \subset F_1 \subset ... \subset F_{h(\phi)}$ be the standard splitting tower of $\phi$. Then it follows from Remark \ref{ffremarks} (2) and the definition that $F_r \sim_F F(\phi^{\times r})$. These equivalences will be used repeatedly in what follows (sometimes implicitly).
\item[$\mathrm{(3)}$] Let $i_0(\phi) = j_0 < j_1 <...<j_t$ be the sequence of defect indices realised by $\phi$ over all possible field extensions of $F$. Contrary to the theory of nondegenerate quadratic forms over fields of characteristic different from 2, it can happen that $t > h(\phi)$. In particular, the formula $j_k = \sum_{r=0}^k i_r(\phi)$ does not hold in general (cf. \citep[Example 7.23]{Hoffmann2}). The difference may be attributed to the fact that the Witt decomposition theorem for nondegenerate quadratic forms over fields of characteristic different from 2 is more subtle than the analogue which we employ here (Proposition \ref{classification}). We still expect however that $j_1 = i_0(\phi) + i_1(\phi)$, as the decomposition theorem plays no role in this case (cf. Corollary \ref{23minimalityi1} for the case where $p \leq 3$). \end{enumerate} \end{remarks}

\begin{lemma}[{cf. \citep[Theorem 7.25]{Hoffmann2}}] \label{isotropytower} Let $\phi$ be a quasilinear $p$-form over $F$. Assume that $\phi$ is not completely split.
\begin{enumerate} \item[$\mathrm{(1)}$] $\mathrm{ndeg}\;\phi_r = \frac{1}{p}\mathrm{ndeg}\;\phi_{r-1}$.
\item[$\mathrm{(2)}$] $\mathrm{dim}\;\phi_r \geq \frac{1}{p}\mathrm{dim}\;\phi_{r-1}$.
\item[$\mathrm{(3)}$] $h(\phi) = \mathrm{log}_p(\mathrm{ndeg}\;\phi)$. \end{enumerate}
\begin{proof} Part (3) follows immediately from part (1). By the construction of the standard splitting tower, statements (1) and (2) follow from Lemma \ref{ffisotropy}. \end{proof} \end{lemma}

Together with Proposition \ref{ndegfunctoriality}, part (3) of Lemma \ref{isotropytower} gives the following result.

\begin{proposition} \label{heightfunctoriality} Let $\phi$ and $\psi$ be anisotropic quasilinear $p$-forms of dimension $>1$ over $F$. If $\phi_{F(\psi)}$ is isotropic, then $h(\psi) \leq h(\phi)$. \end{proposition}

In other words, a rational map $X_\psi \dashrightarrow X_\phi$ of anisotropic quasilinear $p$-hypersurfaces can only exist provided $h(\psi) \leq h(\phi)$. This ``functorial'' property of the standard splitting tower will be studied further in later sections.\\

Now, let $\pi$ be a quasi-Pfister form over $F$. Assume that $\pi$ is not completely split. Then it follows from Lemma \ref{isotropytower} (2) and Lemma \ref{Pfisterisotropy} (3) that $\mathrm{sp}(\phi) = (p^{h(\pi)},p^{h(\pi) - 1},...,p^2,p,1)$. Clearly the same is true of any scalar multiple of $\phi$. In \cite{Hoffmann2}, D. Hoffmann has proved the following result.

\begin{theorem}[{cf. \citep[Theorem 7.14]{Hoffmann2}}] \label{ssppfister} Let $\phi$ be an anisotropic quasilinear $p$-form of dimension $>1$ over $F$. The following conditions are equivalent.
\begin{enumerate} \item[$\mathrm{(1)}$] $\phi$ is similar to a quasi-Pfister form.
\item[$\mathrm{(2)}$] $\mathrm{sp}(\phi) = (p^{h(\phi)}, p^{h(\phi) - 1},...,p^2,p,1)$.
\item[$\mathrm{(3)}$] $\mathrm{dim}\;\phi_1 = \frac{1}{p}\mathrm{dim}\;\phi$. \end{enumerate} \end{theorem}

\begin{remark} By Lemma \ref{isotropytower} (2), the integer $\mathrm{dim}\;\phi_1$ can be no less than $\frac{1}{p}\mathrm{dim}\;\phi$ in the situation of the above theorem. The result therefore says that the minimum value is realised if and only if $\phi$ is similar to a quasi-Pfister form. This is analogous to a classic result of A. Pfister and M. Knebusch in the theory of nondegenerate quadratic forms over fields of characteristic different from 2. \end{remark}

\section{Compressibility of quasilinear $p$-hypersurfaces and some applications}

To proceed further, we will need to recall the main results of \cite{Scully}. The following result shows that if $\psi$ is an anisotropic quasilinear $p$-form of dimension $>1$ over $F$ with $i_1(\psi) = 1$, then the variety $X_\psi$ cannot be ``rationally compressed'' to a quasilinear $p$-hypersurface of smaller dimension.

\begin{proposition}[{\citep[Corollary 5.11]{Scully}}] \label{i1=1incompressibility} Let $\phi$ and $\psi$ be anisotropic quasilinear $p$-forms of dimension $>1$ over $F$, and suppose that there exists a rational map $f \colon X_\psi \dashrightarrow X_\phi$. If $i_1(\psi) = 1$, then $\mathrm{deg}(f) = 1$. \end{proposition}

\begin{corollary} \label{i1=1birational} Let $\phi$ and $\psi$ be anisotropic quasilinear $p$-forms of dimension $>1$ over $F$ such that $\phi_{F(\psi)}$ is isotropic. Assume that $i_1(\psi) = 1$ and $\mathrm{dim}\;\phi \leq \mathrm{dim}\;\psi$. Then
\begin{enumerate} \item[$\mathrm{(1)}$] $X_\psi$ and $X_\phi$ are birational, i.e. $F(\psi) \simeq F(\phi)$. In particular, $\mathrm{dim}\;\phi = \mathrm{dim}\;\psi$.
\item[$\mathrm{(2)}$] $i_1(\phi) = 1$. \end{enumerate}
\begin{proof} Part (1) follows immediately from Proposition \ref{i1=1incompressibility}. If $i_1(\phi) > 1$, then every codimension 1 subform of $\phi$ becomes isotropic over $F(\psi) \simeq F(\phi)$ by Lemma \ref{excellence}. This contradicts (1), so (2) also follows. \end{proof} \end{corollary} 

If $\phi$ is an anisotropic quasilinear $p$-form of dimension $>1$ over $F$, the integer $\mathrm{dim}_{Izh}\phi = \mathrm{dim}\;\phi_1 + 1$ is called the \emph{Izhboldin dimension of $\phi$}. By reducing to the case where $i_1(\psi) = 1$ (that is, where $\mathrm{dim}_{Izh}\psi = \mathrm{dim}\;\psi$), one may deduce from Corollary \ref{i1=1birational} the following result, which was first proved in the case $p=2$ by B. Totaro (cf. \citep[Theorem 5.1]{Totaro}).

\begin{theorem}[{cf. \citep[Theorem 5.12]{Scully}}] \label{compressibilitytheorem} Let $\phi$ and $\psi$ be anisotropic quasilinear $p$-forms of dimension $>1$ over $F$ such that $\phi_{F(\psi)}$ is isotropic.
\begin{enumerate} \item[$\mathrm{(1)}$] $\mathrm{dim}_{Izh}\psi \leq \mathrm{dim}\;\phi$.
\item[$\mathrm{(2)}$] If equality holds in (1), then $\psi_{F(\phi)}$ is isotropic. \end{enumerate} \end{theorem}

Here is a very useful corollary of this result.

\begin{corollary} \label{ffanisotropysubform} Let $\phi$ and $\psi$ be anisotropic quasilinear $p$-forms of dimension $>1$ over $F$, and let $\sigma \subset \phi$ be a subform such that $\mathrm{dim}\;\sigma \leq \mathrm{dim}\;\psi_1$. Then $\sigma_{F(\psi)} \subset (\phi_{F(\psi)})_{an}$. In particular, $\sigma_{F(\psi)}$ is anisotropic.
\begin{proof} Clearly we have $D(\sigma_{F(\psi)}) \subset D(\phi_{F(\psi)})$. In view of Corollary \ref{anisotropicclassification}, it is therefore enough to show that $\sigma_{F(\psi)}$ is anisotropic. If $\mathrm{dim}\;\sigma = 1$, there is nothing to prove. Otherwise, the assertion follows immediately from Theorem \ref{compressibilitytheorem}. \end{proof} \end{corollary}

Another important application of Theorem \ref{compressibilitytheorem} is the following result, which again, is due to B. Totaro in the case where $p=2$ (cf. \citep[Theorem 6.4]{Totaro}).

\begin{theorem}[{\citep[Theorem 7.6]{Scully}}] \label{ruledness} Let $\phi$ and $\psi$ be anisotropic quasilinear $p$-forms of dimension $>1$ over $F$. Assume that $\psi$ is similar to a subform of $\phi$ and that $\mathrm{dim}\;\psi = \mathrm{dim}_{Izh}\phi$. Then $F(\phi)$ is isomorphic to a purely transcendental extension of $F(\psi)$. \end{theorem}

\begin{corollary}[{\citep[Proposition 6.1]{Scully}}] \label{neighbourindex} Let $\phi$ and $\psi$ be anisotropic quasilinear $p$-forms of dimension $>1$ over $F$. Assume that $\psi$ is similar to a subform of $\phi$ and that $\mathrm{dim}\;\psi \geq \mathrm{dim}_{Izh}\;\phi$. Then $\mathrm{dim}\;\psi_1 = \mathrm{dim}\;\phi_1$.
\begin{proof} By Lemma \ref{subformtransitivity}, there exists an $F$-place $F(\phi) \rightharpoonup F(\psi)$. In view of Lemma \ref{placeisotropy}, $i_1(\psi) \geq i_0(\psi_{F(\phi)})$, and hence $\mathrm{dim}\;\psi_1 \leq \mathrm{dim}\;\phi_1$. For the reverse inequality, let $\sigma \subset \psi$ be a subform of dimension $\mathrm{dim}_{Izh}\;\phi$. By the same reasoning we have $\mathrm{dim}\;\sigma_1 \leq \mathrm{dim}\;\psi_1$. We are therefore reduced to the case where $\mathrm{dim}\;\psi = \mathrm{dim}_{Izh}\phi$. In this case, $F(\psi) \sim_F F(\phi)$ by Theorem \ref{ruledness} and Example \ref{stabbireq}. By Lemma \ref{placeisotropy}, we have $i_1(\psi) = i_0(\psi_{F(\phi)})$. But $i_0(\psi_{F(\phi)}) \leq 1$ by Corollary \ref{ffanisotropysubform}, and so $i_1(\psi) = 1$. The result follows. \end{proof} \end{corollary}

The above results have a number of important consequences for the standard splitting pattern. We will need the following proposition.

\begin{proposition} \label{equivalenceandssp} Let $\phi$ and $\psi$ be anisotropic quasilinear $p$-forms of dimension $>1$ over $F$. If $F(\psi) \sim_F F(\phi)$, then $\widetilde{\mathrm{sp}}(\psi) = \widetilde{\mathrm{sp}}(\phi)$.
\begin{proof} Let $\sigma \subset \psi$ be a subform of dimension $\mathrm{dim}_{Izh}\psi$. By Theorem \ref{ruledness}, $F(\psi)$ is isomorphic to a purely transcendental extension of $F(\sigma)$. It follows from Example \ref{stabbireq} and Remark \ref{sspremarks} (2) that the standard splitting patterns of $\psi$ and $\sigma$ are $F$-equivalent. In particular, the standard splitting tower of $\psi$ can be used to compute $\mathrm{sp}(\sigma)$ by Lemma \ref{placeisotropy}. Now, by Corollary \ref{neighbourindex}, we have $\mathrm{dim}\;\psi_1 = \mathrm{dim}\;\sigma_1$. It follows that $(\sigma_1)_{F(\psi)} \simeq \psi_1$, and in view of the above remarks, we have
\begin{equation*} \widetilde{\mathrm{sp}}(\psi) = \mathrm{sp}(\psi_1) = \mathrm{sp}((\sigma_1)_{F(\psi)}) = \widetilde{\mathrm{sp}}(\sigma). \end{equation*}
Since $F(\psi) \sim_F F(\sigma)$, we may therefore replace $\psi$ with $\sigma$ and reduce to the case $i_1(\psi) = 1$. In the same way, we may also reduce to the case where $i_1(\phi) = 1$. Now, Lemma \ref{placeisotropy} (1) shows that both $\phi_{F(\psi)}$ and $\psi_{F(\phi)}$ are isotropic. By Corollary \ref{i1=1birational} (1), we have $F(\psi) \simeq F(\phi)$. In particular, $\mathrm{dim}\;\psi = \mathrm{dim}\;\phi$. Since $i_1(\psi) = i_1(\phi) = 1$, it follows that $\mathrm{dim}\;\psi_1 = \mathrm{dim}\;\phi_1$. If $h(\psi) = 1$, then we are done. If $h(\psi) >1$, then we have to show that $\mathrm{sp}(\psi_2) = \mathrm{sp}(\phi_2)$. We may identify the fields $L = F(\psi) = F(\phi)$. Under this identification, Remark \ref{sspremarks} (2) implies that $L(\psi_1) \sim_L L(\phi_1)$. Since $h(\psi_1) < h(\psi)$ we may argue by induction to get
\begin{equation*} \mathrm{sp}(\psi_2) = \widetilde{\mathrm{sp}}(\psi_1) = \widetilde{\mathrm{sp}}(\phi_1) = \mathrm{sp}(\phi_2), \end{equation*}
and the proposition is proved. \end{proof} \end{proposition}

In particular, the standard splitting pattern is a birational invariant on the class of anisotropic quasilinear $p$-hypersurfaces. We shall explore this further in the following sections (\S 8 in particular).

\section{Standard splitting and subforms} 

We now use the results of the previous section to study the relationship between the standard splitting pattern of a quasilinear $p$-form and those of its subforms. We begin by studying a special class of subforms, which we call neighbours.

\subsection{Neighbours.} The following definition is motivated by Corollary \ref{neighbourindex}.

\begin{definition} \label{neighbourdef} Let $\phi$ and $\psi$ be anisotropic quasilinear $p$-forms of dimension $>1$ over $F$. We say that $\psi$ is a \emph{neighbour of $\phi$} if $\psi$ is similar to a subform of $\phi$ and $\mathrm{dim}\;\psi \geq \mathrm{dim}_{Izh}\phi$. If additionally $\mathrm{dim}\;\psi = \mathrm{dim}_{Izh}\phi$, then $\psi$ is called a \emph{minimal neighbour of $\phi$}. \end{definition}

\begin{proposition} \label{sspneighbours} Let $\phi$ be an anisotropic quasilinear $p$-form of dimension $>1$ over $F$, and let $\psi$ be a neighbour of $\phi$. Then $F(\phi)$ is isomorphic to a purely transcendental extension of $F(\psi)$. In particular, $F(\psi) \sim_F F(\phi)$ and $\widetilde{\mathrm{sp}}(\psi) = \widetilde{\mathrm{sp}}(\phi)$.
\begin{proof} The second statement follows from the first in view of Example \ref{stabbireq} and Proposition \ref{equivalenceandssp}. For the first statement, the case where $\psi$ is a minimal neighbour of $\phi$ was treated in Theorem \ref{ruledness}. To prove the general case, one now only needs to check that every minimal neighbour of $\psi$ is also a minimal neighbour of $\phi$. This was done in Corollary \ref{neighbourindex}. \end{proof} \end{proposition}

The following class of forms are of special importance.

\begin{definition}[{cf. \citep[Definition 4.12]{Hoffmann2}}] Let $\phi$ be a quasilinear $p$-form of dimension $>1$ over $F$. We say that $\phi$ is a \emph{quasi-Pfister neighbour} if there exists a quasi-Pfister form $\pi$ over $F$ such that $\phi$ is similar to a subform of $\pi$ and $p \mathrm{dim}\;\phi > \mathrm{dim}\;\pi$. \end{definition}

In the above definition, if $\phi$ is anisotropic, then $\phi$ is a neighbour of $\pi$ in the sense of Definition \ref{neighbourdef} (this follows from Theorem \ref{ssppfister}). Note that in this case, we must have $\pi \simeq \phi_{qp}$ by Lemma \ref{normformlemma}. We can now give the following description of anisotropic quasi-Pfister neighbours.

\begin{theorem} \label{pfisterneighbourclassification} Let $\phi$ be an anisotropic quasilinear $p$-form of dimension $>1$ over $F$, and let $n$ be the unique nonnegative integer such that $p^n < \mathrm{dim}\;\phi \leq p^{n+1}$. The following conditions are equivalent.
\begin{enumerate} \item[$\mathrm{(1)}$] $\phi$ is a quasi-Pfister neighbour.
\item[$\mathrm{(2)}$] $\phi$ is a neighbour of $\phi_{qp}$.
\item[$\mathrm{(3)}$] $\phi_{F(\phi_{qp})}$ is isotropic.
\item[$\mathrm{(4)}$] $h(\phi) = n+1$.
\item[$\mathrm{(5)}$] $\widetilde{\mathrm{sp}}(\phi) = (p^n,p^{n-1},...,p^2,p,1)$.
\item[$\mathrm{(6)}$] $\phi_1$ is similar to a quasi-Pfister form. \end{enumerate}
\begin{proof} We have already discussed the equivalence of (1) and (2) above. Moreover, by Lemma \ref{normformlemma}, (2) holds if and only if $\mathrm{ndeg}\;\phi = \mathrm{dim}\;\phi_{qp} = p^{n+1}$. The equivalence of (2) and (4) therefore follows from part (3) of Lemma \ref{isotropytower}. Since $\widetilde{\mathrm{sp}}(\phi) = \mathrm{sp}(\phi_1)$, the equivalence of (5) and (6) follows from Theorem \ref{ssppfister}. The implication $(5) \Rightarrow (4)$ is trivial, and (2) implies (3) and (5) by Proposition \ref{sspneighbours}. Finally, suppose that $h(\phi) > n+1$. Then $\mathrm{dim}\;\phi_{qp} \geq p^{n+2}$, so that $\mathrm{dim}_{Izh}\phi_{qp} > p^{n+1}$. Since $\mathrm{dim}\;\phi \leq p^{n+1}$, $\phi_{F(\phi_{qp})}$ is anisotropic by Theorem \ref{compressibilitytheorem} (1). This shows that (3) implies (4), and the proof is complete. \end{proof} \end{theorem}

\begin{remark} If $\phi$ is an anisotropic quasilinear $p$-form of dimension $>1$ over $F$, and $n$ is the unique nonnegative integer satisfying $p^n < \mathrm{dim}\;\phi \leq p^{n+1}$, then it follows from Lemma \ref{isotropytower} that $h(\phi) \geq n+1$. Therefore, anisotropic quasi-Pfister neighbours may be interpreted as precisely those anisotropic forms which have ``smallest possible height''. For example, every anisotropic form of height 1 is a quasi-Pfister neighbour. In particular, if $\phi$ is an anisotropic quasilinear $p$-form of dimension $>1$ over $F$, then $\phi_{h(\phi) - 1}$ is a quasi-Pfister neighbour. The following definition therefore makes sense. \end{remark}

\begin{definition} \label{quasiPfisterheight} Let $\phi$ be an anisotropic quasilinear $p$-form of dimension $>1$ over $F$. The \emph{quasi-Pfister height of $\phi$}, denoted $h_{qp}(\phi)$, is defined to be the smallest nonnegative integer $r$ such that the form $\phi_r$ is a quasi-Pfister neighbour. \end{definition}

\subsection{General subforms.} We now consider the case of general subforms. It is worth writing down the following observation.

\begin{lemma} \label{heightext} Let $\phi$ be an anisotropic quasilinear $p$-form of dimension $>1$ over $F$, and let $L$ be a field extension of $F$. Let $(F_r)$ denote the standard splitting tower of $\phi$. The following conditions are equivalent.
\begin{enumerate} \item[$\mathrm{(1)}$] $h(\phi_L) < h(\phi)$.
\item[$\mathrm{(2)}$] $\phi_r$ becomes isotropic over $L \cdot F_s$ for some $s \in [0,h(\phi))$. \end{enumerate}
\begin{proof} By Remark \ref{sspremarks} (2), the standard splitting tower of $\phi_L$ is $L$-equivalent to $(L \cdot F_r)$ (this tower having possibly shorter length than $(F_r)$). By Lemma \ref{placeisotropy}, the latter tower determines the standard splitting pattern of $\phi_L$. In particular, if $\phi_r$ remains anisotropic over $L \cdot F_r$ for each $r$, then $h(\phi_L) = h(\phi)$. Conversely, suppose that some $\phi_r$ becomes isotropic over $L \cdot F_r$. Choose $s$ to be minimal among all $r$ with this property. Then, by Corollary \ref{normdegcriterion} and Lemma \ref{isotropytower} (3), we have
\begin{equation*} h(\phi_L) = s + h((\phi_s)_{L \cdot F_s}) < s + h(\phi_s) = h(\phi), \end{equation*}
as we wanted. \end{proof} \end{lemma}

In view of Lemma \ref{heightext}, the following statement makes sense.

\begin{lemma} \label{ffsubformtower} Let $\phi$ and $\psi$ be anisotropic quasilinear $p$-forms of dimension $>1$ over $F$. Assume that $\psi$ is similar to a subform of $\phi$, and that $\psi_{F(\phi)}$ is anisotropic. Let $(F_r)$ be the standard splitting tower of $\psi$. Then either
\begin{enumerate} \item[$\mathrm{(1)}$] $h(\psi_{F(\phi)}) = h(\psi)$ and $\mathrm{sp}(\psi_{F(\phi)}) = \mathrm{sp}(\psi)$, or
\item[$\mathrm{(2)}$] $h(\psi_{F(\phi)}) = h(\psi)-1$ and
\begin{equation*} \mathrm{sp}(\psi_{F(\phi)}) = (\mathrm{dim}\;\psi, \mathrm{dim}\;\psi_1,...,\mathrm{dim}\;\psi_{s-1},\mathrm{dim}\;\psi_{s+1},...,\mathrm{dim}\;\psi_{h(\psi)}=1), \end{equation*}
where $s \in [1,h(\psi))$ is the least positive integer $r$ such that $\psi_r$ becomes isotropic over $F_r(\phi)$. \end{enumerate}
\begin{proof} As in the proof of Lemma \ref{heightext}, the tower $(F_r(\phi))$ computes the standard splitting pattern of $\psi_{F(\phi)}$. If $h(\psi_{F(\phi)}) = h(\psi)$, then $\mathrm{sp}(\psi_{F(\phi)}) = \mathrm{sp}(\psi)$ by Lemma \ref{heightext}. If $h(\psi_{F(\phi)}) < h(\psi)$, then some $\psi_r$ becomes isotropic over $F_r(\phi)$ by the same result. Choose $s \in [1,h(\psi))$ minimal among  those $r$ with this property, and let $\sigma = (\phi_{F_s})_{an}$. By assumption, there is $a \in F^*$ such that $\psi \subset a\phi$. It follows that $D(\psi_s) = D(\psi_{F_s}) \subset D(a\phi_{F_s}) = D(a\sigma)$, and so $\psi_s$ is similar to a subform of $\sigma$ by Corollary \ref{anisotropicclassification}. Now, by Remark \ref{ffremarks} (2), we have $F_s(\sigma) \sim_{F_s} F_s(\phi)$. In particular, $(\psi_s)_{F_s(\sigma)}$ is isotropic by Lemma \ref{placeisotropy}. By Theorem \ref{compressibilitytheorem}, we therefore have $\mathrm{dim}_{Izh}\sigma \leq \mathrm{dim}\;\psi_s$. This shows that $\psi_s$ is a neighbour of $\sigma$. In particular, $F_{s+1} = F_s(\psi_s) \sim_{F_s} F_s(\sigma) \sim_{F_s} F_s(\phi)$ by Proposition \ref{sspneighbours}. By Remark \ref{sspremarks} (2), the standard splitting tower of $(\psi_{F(\phi)})_s$ is therefore $F_s$-equivalent to $F_{s+1} \subset F_{s+2} \subset ... \subset F_{h(\psi)}$. By the minimality of $s$, it follows that
\begin{equation*} \mathrm{sp}(\psi_{F(\phi)}) = (\mathrm{dim}\;\psi, \mathrm{dim}\;\psi_1,...,\mathrm{dim}\;\psi_{s-1},\mathrm{dim}\;\psi_{s+1},...,\mathrm{dim}\;\psi_{h(\psi)}=1), \end{equation*}
and the lemma is proved. \end{proof} \end{lemma}

\begin{remark} We expect that the condition ``$\psi$ is similar to a subform of $\phi$'' in Lemma \ref{ffsubformtower} can be replaced with the weaker condition ``$\phi_{F(\psi)}$ is isotropic''. In \S 8, we will prove this in the special case where $p \leq 3$ (cf. Proposition \ref{functorialityprop}). \end{remark}

Now we can prove the following result, which (applied repeatedly) gives a description of the general relationship between the standard splitting pattern of an anisotropic form and those of its subforms.

\begin{proposition} \label{sspsubforms} Let $\phi$ and $\psi$ be anisotropic quasilinear $p$-forms of dimension $>1$ over $F$. Assume that $\psi$ is similar to a codimension 1 subform of $\phi$, and let $(F_r)$ be the standard splitting tower of $\psi$. Then either
\begin{enumerate} \item[$\mathrm{(1)}$] $h(\phi) = h(\psi) + 1$, $i_1(\phi) = 1$ and $\widetilde{\mathrm{sp}}(\phi) = \mathrm{sp}(\psi)$, or
\item[$\mathrm{(2)}$] $h(\phi) = h(\psi)$, $i_1(\phi) > 1$ and $\widetilde{\mathrm{sp}}(\phi) = \widetilde{\mathrm{sp}}(\psi)$, or
\item[$\mathrm{(3)}$] $h(\phi) = h(\psi)$, $i_1(\phi) = 1$ and
\begin{equation*} \widetilde{\mathrm{sp}}(\phi) = (\mathrm{dim}\;\psi,\mathrm{dim}\;\psi_1,...,\mathrm{dim}\;\psi_{s-1},\mathrm{dim}\;\psi_{s+1},...,\mathrm{dim}\;\psi_{h(\psi)} = 1), \end{equation*}
where $s \in [1,h(\psi))$ is the least positive integer $r$ such that $\psi_r$ becomes isotropic over $F_r(\phi)$. \end{enumerate}
\begin{proof} Suppose first that $i_1(\phi) > 1$. Then $\psi$ is a neighbour of $\phi$, and so $\widetilde{\mathrm{sp}}(\phi) = \widetilde{\mathrm{sp}}(\psi)$ by Proposition \ref{sspneighbours}. This is case (2). We may therefore assume that $i_1(\phi) = 1$. In this case, we have $\phi_1 \simeq \psi_{F(\phi)}$ by Corollary \ref{ffanisotropysubform}. Since $\widetilde{\mathrm{sp}}(\phi) = \mathrm{sp}(\phi_1)$ (cf. Remark \ref{sspremarks} (1)), we are either in case (1) or case (3) by Lemma \ref{ffsubformtower}. \end{proof} \end{proposition}

We have seen in Proposition \ref{sspneighbours} that the standard splitting pattern of an anisotropic quasilinear $p$-form is determined completely by that of any of its neighbours. Now we can describe its relation to the the standard splitting patterns of the next class of subforms, namely those of dimension $\mathrm{dim}\;\phi_1$.

\begin{corollary} \label{sspsubformsdim1} Let $\phi$ be an anisotropic quasilinear $p$-form of dimension $>1$ over $F$, and let $\psi \subset \phi$ be a subform of dimension $\mathrm{dim}\;\phi_1$. Let $(F_r)$ be the standard splitting tower of $\psi$. Then either
\begin{enumerate} \item[$\mathrm{(1)}$] $h(\phi) = h(\psi) + 1$, $i_1(\phi) = 1$ and $\widetilde{\mathrm{sp}}(\phi) = \mathrm{sp}(\psi)$, or
\item[$\mathrm{(2)}$] $h(\phi) = h(\psi)$ and 
\begin{equation*} \widetilde{\mathrm{sp}}(\phi) = (\mathrm{dim}\;\psi,\mathrm{dim}\;\psi_1...,\mathrm{dim}\;\psi_{s-1},\mathrm{dim}\;\psi_{s+1},...,\mathrm{dim}\;\psi_{h(\psi)} = 1) \end{equation*}
where $s \in [1,h(\psi))$ is the least positive integer $r$ such that $\psi_r$ becomes isotropic over $F_r(\phi)$. \end{enumerate}
\begin{proof} By Proposition \ref{sspneighbours}, we may replace $\phi$ with a minimal neighbour of itself. Everything then follows from Proposition \ref{sspsubforms}. \end{proof} \end{corollary}

\begin{remark} Proposition \ref{sspsubforms} and Corollary \ref{sspsubformsdim1} show that the standard splitting patterns of two subforms of small codimension in a given form should not be markedly different. One can hope to use this observation to obtain restrictions on the standard splitting pattern of the ambient form. \end{remark}

\section{Two comparison results} 

In this section we prove some results which allow us to make a comparison between the standard splitting pattern (or at least some component of it) of a quasilinear $p$-form over $F$ and the standard splitting patterns of the same form over some field extensions of $F$. These results in turn give restrictions on the possible values of the standard splitting pattern over the base field.

\subsection{First comparison result.} In this first subsection, we compare the standard splitting pattern of a quasilinear $p$-form $\phi$ over $F$ to its standard splitting pattern over the field $F(\phi_{qp})$. The precise result is the following.

\begin{proposition} \label{firstcomparison} Let $\phi$ be an anisotropic quasilinear $p$-form of dimension $>1$ over $F$, and let $s = h_{qp}(\phi)$ (cf. Definition \ref{quasiPfisterheight}). Assume that $\phi$ is not a quasi-Pfister neighbour (in other words, $s > 0$). Then
\begin{enumerate} \item[$\mathrm{(1)}$] $\mathrm{sp}(\phi_{F(\phi_{qp})}) = (\mathrm{dim}\;\phi, \mathrm{dim}\;\phi_1,..., \mathrm{dim}\;\phi_{s-1},\mathrm{dim}\;\phi_{s+1},...,\mathrm{dim}\;\phi_{h(\phi)} = 1)$.
\item[$\mathrm{(2)}$] $h_{qp}(\phi_{F(\phi_{qp})})< s$. \end{enumerate}
\begin{proof} Given part (1), statement (2) follows immediately from Theorem \ref{pfisterneighbourclassification} (for example, from the equivalence of parts (1) and (5)). For (1), let $(F_r)$ be the standard splitting tower of $\phi$. Since $\phi$ is not a Pfister neighbour, $\phi_{F(\phi_{qp})}$ is anisotropic by Theorem \ref{pfisterneighbourclassification}. On the other hand, we have $h(\phi_{F(\phi_{qp})}) < h(\phi)$ by Lemma \ref{normformext}. Since $\phi$ is similar to a subform of $\phi_{qp}$ (cf. Lemma \ref{normformlemma}), we can therefore apply Lemma \ref{ffsubformtower} to see that
\begin{equation*} \mathrm{sp}(\phi_{F(\phi_{qp})}) = (\mathrm{dim}\;\phi, \mathrm{dim}\;\phi_1,..., \mathrm{dim}\;\phi_{s-1},\mathrm{dim}\;\phi_{s+1},...,\mathrm{dim}\;\phi_{h(\phi)} = 1), \end{equation*}
where $s \in [1,h(\phi))$ is the least positive integer $r$ such that $\phi_r$ becomes isotropic over $F_r(\phi_{qp})$. We need to show that $s = h_{qp}(\phi)$. Equivalently, we must show that $\phi_s$ is a quasi-Pfister neighbour, but $\phi_r$ is not a quasi-Pfister neighbour for any $r < s$. But, for any $r$, we have $F_r(\phi_{qp}) \sim_{F_r} F_r((\phi_r)_{qp})$ by Lemma \ref{normformext} and Remark \ref{ffremarks} (2). In particular, $\phi_r$ can become isotropic over $F_r(\phi_{qp})$ if and only if $\phi_r$ is a quasi-Pfister neighbour by Theorem \ref{pfisterneighbourclassification} and Lemma \ref{placeisotropy}. The result follows. \end{proof} \end{proposition}

By a repeated application of Proposition \ref{firstcomparison}, we obtain the following corollary, which suggests the possibility of an ``inductive'' approach to the study of the higher defect indices.

\begin{corollary} \label{shiftinglemma} Let $\phi$ be an anisotropic quasilinear $p$-form of dimension $>1$ over $F$. Assume that $\phi$ is not a quasi-Pfister neighbour (that is, $h_{qp}(\phi)>0$), and let $s \in [1,h_{qp}(\phi)]$. Then there exists a field extension $\widetilde{F}$ of $F$ such that
\begin{enumerate} \item[$\mathrm{(1)}$] $\phi_{\widetilde{F}}$ is anisotropic.
\item[$\mathrm{(2)}$] $i_r(\phi_{\widetilde{F}}) = i_r(\phi)$ for all $r < s$.
\item[$\mathrm{(3)}$] $h_{qp}(\phi_{\widetilde{F}}) < s$. \end{enumerate} \end{corollary}

As a first step, we obtain the following result which was already proved in \cite{Scully}.

\begin{corollary}[{\citep[Proposition 6.6]{Scully}}] \label{twistedpfister} Let $\phi$ be an anisotropic quasilinear $p$-form over $F$. Then there exists a field extension $\widetilde{F}$ of $F$ such that $\phi_{\widetilde{F}}$ is an anisotropic quasi-Pfister neighbour.
\begin{proof} If $\phi$ is a quasi-Pfister neighbour, there is nothing to prove. Otherwise, the claim follows from Corollary \ref{shiftinglemma} by taking $s =1$. \end{proof} \end{corollary}

In particular, we recall that this corollary gives the following restriction on the possible values of the invariant $i_1$, which in the case where $p=2$ was first proved by D. Hoffmann and A. Laghribi in \cite{HoffmannLaghribi2}.

\begin{corollary}[{\citep[Corollary 6.8]{Scully}}] \label{Hoffmannbound} Let $\phi$ be an anisotropic quasilinear $p$-form of dimension $>1$ over $F$, and write $\mathrm{dim}\;\phi = p^n + m$ for uniquely determined integers $n \geq 0$ and $m \in [1,p^{n+1} - p^n]$. Then $i_1(\phi) \leq m$.
\begin{proof} Let $\widetilde{F}$ be as in Corollary \ref{twistedpfister}. Then $i_0(\phi_{\widetilde{F}(\phi)}) = i_1(\phi_{\widetilde{F}}) = m$ by Theorem \ref{pfisterneighbourclassification}. Since $F(\phi)$ is a subfield of $\widetilde{F}(\phi)$, the result follows immediately. \end{proof} \end{corollary}

In view of Corollary \ref{Hoffmannbound}, it makes sense to introduce the following definition.

\begin{definition}[{cf. \citep[\S 7.6]{Hoffmann2}}] Let $\phi$ be an anisotropic quasilinear $p$-form of dimension $>1$ over $F$, and write $\mathrm{dim}\;\phi = p^n + m$ for uniquely determined integers $n \geq 0$ and $m \in [1,p^{n+1} - p^n]$. We say that $\phi$ has \emph{maximal splitting} if $i_1(\phi) = m$. \end{definition}

By Theorem \ref{pfisterneighbourclassification}, any anisotropic quasi-Pfister neighbour has maximal splitting. It is an interesting problem to determine conditions under which the converse holds. For example, Theorem \ref{ssppfister} shows that the converse holds for forms whose dimension is a power of $p$. In \cite{Hoffmann2}, it is suggested that the property of being an anisotropic quasi-Pfister neighbour is completely characterised by the property of maximal splitting for forms whose dimension is ``sufficiently close to a power of $p$''. More precisely, we ask the following questions.

\begin{questions}[{cf. \citep[Remark 7.32]{Hoffmann2}}] \label{maxsplittingquestions} Let $\phi$ be an anisotropic quasilinear $p$-form of dimension $>p$ over $F$, and write $\mathrm{dim}\;\phi = p^n + m$ for uniquely determined integers $n \geq 1$ and $m \in [1,p^{n+1} - p^n]$. Suppose that $\phi$ has maximal splitting.
\begin{enumerate} \item[$\mathrm{(1)}$] Assume that $p=2$ and $\mathrm{dim}\;\phi > 5$. If $m > 2^{n-2}$, must $\phi$ be a quasi-Pfister neighbour?
\item[$\mathrm{(2)}$] Assume that $p>2$. If $m>p^{n-1}$, must $\phi$ be a quasi-Pfister neighbour? \end{enumerate} \end{questions}

\begin{remark} In all dimensions not considered in Questions \ref{maxsplittingquestions} (1) and (2), one can easily construct anisotropic forms with maximal splitting which are not quasi-Pfister neighbours (cf. \citep[Example 7.31]{Hoffmann2}). A positive answer in either case would therefore give a complete solution to the problem of finding dimension-theoretic conditions under which the property of being a quasi-Pfister neighbour is completely characterised by maximal splitting. \end{remark}

Note here that Question \ref{maxsplittingquestions} is a direct analogue of Conjecture \ref{maxsplittingconjecture}. While the latter conjecture remains wide open, there is substantial evidence for its truth, most notably in the motivic approach of A. Vishik (cf. \cite{IzhboldinVishik}). By direct analogy, it therefore seems reasonable to expect a positive answer to Question \ref{maxsplittingquestions} (1) (and in fact, we will show later that this is indeed the case). The evidence for Question \ref{maxsplittingquestions} (2) is much weaker. Nevertheless, we can now prove the following partial result.

\begin{proposition} \label{maxsplittingallp} Let $\phi$ be an anisotropic quasilinear $p$-form over $F$ such that $p^{n+1} - p^{n-1} \leq \mathrm{dim}\;\phi \leq p^{n+1}$ for some positive integer $n$. If $\phi$ has maximal splitting, then $\phi$ is a quasi-Pfister neighbour. \end{proposition}

First, we will need a lemma.

\begin{lemma} \label{maxsplittinglemma} Let $\phi$ be an anisotropic quasilinear $p$-form of dimension $>1$ over $F$, and let $h = h(\phi)$. After replacing $F$ by a purely transcendental extension of itself, there exist elements $a_1,...,a_h \in F$ such that
\begin{enumerate} \item[$\mathrm{(1)}$] The quasi-Pfister form $\pfister{a_1,...,a_h}$ over $F$ is anisotropic.
\item[$\mathrm{(2)}$] $i_0(\phi_{F_{a_i}}) = i_1(\phi)$ for all $i \in [1,h]$. \end{enumerate}
\begin{remark} By Lemma \ref{sepext}, the integer $i_1(\phi)$ is stable under making purely transcendental extensions of $F$. \end{remark}
\begin{proof} Suppose that for $r \in [0,h)$, we have constructed (after possibly replacing $F$ by a purely transcendental extension of itself) elements $a_1,...,a_r \in F$ such that
\begin{itemize} \item $\pi = \pfister{a_1,...,a_r}$ is anisotropic.
\item $i_0(\phi_{F_{a_i}}) = i_1(\phi)$ for all $i \in [1,r]$. \end{itemize}
Let us construct the next element $a_{r+1}$. First, by Remark \ref{ffremarks} (1), we can write $F(\phi) = F(X)_{u}$, where $F(X)$ is a purely transcendental extension of $F$, and $u \in F(X) \setminus F(X)^p$. Since $F(X)_{a_i}$ is a purely transcendental extension of $F_{a_i}$, we have $i_0(\phi_{F(X)_{a_i}}) = i_0(\phi_{F_{a_i}}) = i_1(\phi)$ for all $i \leq r$ by Lemma \ref{sepext}. Moreover, $i_0(\phi_{F(X)_u}) = i_1(\phi)$ by definition. If we can show that the quasi-Pfister form $\pi_{F(X)} \otimes \pfister{u}$ is anisotropic, then we may replace $F$ by $F(X)$, put $a_{r+1} = u$, and conclude by induction. So, suppose that $\pi_{F(X)} \otimes \pfister{u}$ is isotropic. Then $\pi_{F(\phi)}$ is isotropic by Lemma \ref{pinsepisotropy} (1). But by Proposition \ref{heightfunctoriality}, this implies that $h = h(\phi) \leq h(\pi) = r$, contradicting the choice of $r$. The lemma is proved. \end{proof} \end{lemma}

\begin{proof}[Proof of Proposition \ref{maxsplittingallp}] By Theorem \ref{pfisterneighbourclassification}, the property of being a quasi-Pfister neighbour depends only on the height of the given form. Since the height is invariant under separable extensions (cf. Lemmas \ref{sepext} (2) and \ref{isotropytower} (3)), we are free to make purely transcendental extensions of the base field. By Lemma \ref{maxsplittinglemma}, we may therefore assume that there exist $a_1,...,a_h \in F$ (where $h = h(\phi)$) such that
\begin{itemize} \item $\pi = \pfister{a_1,...,a_h}$ is anisotropic.
\item $i_0(\phi_{F_{a_i}}) = i_1(\phi)$ for all $i \in [1,h]$. \end{itemize}
Now, write $\mathrm{dim}\;\phi = p^{n+1} - p^{n-1} + s$ for some nonnegative integer $s$. Since $\phi$ has maximal splitting, $i_1(\phi) = (p^{n+1} - p^{n-1} - p^n + s)$. Since $(p^2 - p - 1)>(p^2 - 2p)$, we have
\begin{align*} (p^2 - p - 1)i_0(\phi_{F_{a_i}}) - (p^2 - p)\mathrm{dim}\;\phi &= (p^2 - p - 1)(p^{n+1} - p^{n-1} - p^n +s)\\
                     &\;\;\;\; - (p^2-2p)(p^{n+1} - p^{n-1} + s)\\
                     & \geq  (p^2 - p - 1)(p^{n+1} - p^{n-1} - p^n)\\
                     &\;\;\;\; -(p^2 - 2p)(p^{n+1} - p^{n-1})\\
                     &= p^{n-1} \end{align*}
for all $i \in [1,h]$. By Lemma \ref{weirdlemma}, for each $i$ we have a form $\tau_i$ of dimension $p^{n-1}$ such that $\psi_i = \pfister{a_i} \otimes \tau_i \subset \phi$. Since $\mathrm{dim}\;\phi_1 = p^n = \mathrm{dim}\;\psi_i$, we have $(\psi_i)_{F(\phi)} \simeq \phi_1$ for all $i$ by Corollary \ref{ffanisotropysubform}. In particular, $\phi_1$ is divisible by $\pfister{a_i}_{F(\phi)}$ for all $i$. Since $D(\pfister{a_i}_{F(\phi)})$ is a field (cf. Lemma \ref{Pfisterisotropy} (1)), we therefore have $a_iD(\phi_1) = D(\phi_1)$ for all $i$. If additionally $1 \in D(\phi_1)$, then it follows that $F(\phi)^p(a_1,...,a_h) \subset D(\phi_1)$. In general, we can multiply through by a scalar to get $F(\phi)^p(a_1,...,a_h) \subset aD(\phi_1)$ for some $a \in F(\phi)^*$. By Corollary \ref{anisotropicclassification}, this means that $\pi_{F(\phi)}$ is similar to a subform of $\phi_1$. Since both forms have the same dimension, $\phi_1$ is in fact similar to $\pi_{F(\phi)}$. But if $\phi_1$ is similar to a quasi-Pfister form, then $\phi$ must itself be a quasi-Pfister neighbour by Theorem \ref{pfisterneighbourclassification}. This completes the proof. \end{proof}

\begin{remark} Proposition \ref{maxsplittingallp} improves \citep[Corollary 7.29]{Hoffmann2}, which required the stronger assumption that $p^{n+1} - p < \mathrm{dim}\;\phi \leq p^{n+1}$. We will give a positive answer to Question \ref{maxsplittingquestions} (1) in full generality in \S 9. \end{remark}

\subsection{Second comparison result.} We now apply the results of \S 5 in a slightly different direction to produce further comparison results for higher defect indices. We begin by remarking that if $\phi$ and $\psi$ are two quasilinear $p$-forms of dimension $>1$ over $F$, and if the field $F(\phi \times \psi)$ is defined, then it is canonically isomorphic to both $F(\phi)(\psi)$ and $F(\psi)(\phi)$. This fact will be employed frequently below. The basic observation is the following.

\begin{proposition} \label{secondcomparisonbasic} Let $\phi$, $\psi$ and $\sigma$ be anisotropic quasilinear $p$-forms of dimension $>1$ over $F$. Then
\begin{equation*} i_0(\phi_{F(\psi \times \sigma)}) - i_0(\phi_{F(\psi)}) \geq \mathrm{min} \lbrace i_0(\phi_{F(\sigma)}), [\frac{\mathrm{dim}\;\psi_1}{p}]\rbrace. \end{equation*}
\begin{proof} By Lemma \ref{sepext}, the whole statement is stable under making purely transcendental extensions of $F$. In view of Remark \ref{ffremarks} (1), we may therefore assume that $\sigma = \form{1,a}$ for some $a \in F^*$, so that $F(\sigma) = F_a$. Let $s$ denote the integer on the right hand side of the inequality in the statement of the Proposition. Then, in particular, we have $i_0(\phi_{F(\sigma)}) \geq s$. By Lemma \ref{pinsepisotropysubform}, there exists a subform $\tau \subset \phi$ such that $\mathrm{dim}\;\tau \leq ps$ and $i_0(\tau_{F(\sigma)}) \geq s$. On the other hand, by the definition of $s$ we have
\begin{equation*} \mathrm{dim}\;\tau \leq ps \leq \mathrm{dim}\;\psi_1, \end{equation*}
and hence $\tau_{F(\psi)} \subset (\phi_{F(\psi)})_{an}$ by Corollary \ref{ffanisotropysubform}. Let $\eta = (\phi_{F(\psi)})_{an}$. Then (using the remark at the beginning of this subsection)
\begin{align*} i_0(\phi_{F(\psi \times \sigma)}) - i_0(\phi_{F(\psi)}) &= i_0(\phi_{F(\psi)}) + i_0(\eta_{F(\psi)(\sigma)}) - i_0(\phi_{F(\psi)})\\
&= i_0(\eta_{F(\psi)(\sigma)})\\
& \geq i_0(\tau_{F(\psi)(\sigma)})\\
& = i_0(\tau_{F(\sigma)(\psi)})\\
& \geq i_0(\tau_{F(\sigma)}) \geq s, \end{align*}
and the result follows. \end{proof} \end{proposition}

We can now prove the main result of this subsection.

\begin{theorem} \label{secondcomparison} Let $\phi$ be an anisotropic quasilinear $p$-form of dimension $>1$ over $F$, and let $L$ be a field extension of $F$. Then
\begin{equation*} i_0(\phi_{L(\phi)}) - i_1(\phi) \geq \mathrm{min}\lbrace i_0(\phi_L), [\frac{\mathrm{dim}\;\phi_1}{p}] \rbrace. \end{equation*}
\begin{proof} In view of Lemma \ref{sepext}, it is enough to treat the case where $F \subset L$ is a purely inseparable algebraic extension. We may clearly reduce further to the case where $F \subset L$ is finite. We will proceed by induction on $[L:F]$, which is a power of $p$. The case where $L = F$ is trivial. Otherwise, consider a filtration $F \subset M \subset L$, where $M$ is a subfield of $L$ satisfying $[M:F] = p$. The field $M$ may be regarded as the function field of a 0-dimensional quasilinear $p$-hypersurface. In particular, we can apply Proposition \ref{secondcomparisonbasic} in the case where $\psi = \phi$ and $F(\sigma) = M$ to obtain
\begin{equation} \label{eq1} i_0(\phi_{M(\phi)}) - i_1(\phi) \geq \mathrm{min}\lbrace i_0(\phi_M),[\frac{\mathrm{dim}\;\phi_1}{p}]\rbrace. \end{equation}
If $M = L$, then we are done. Suppose that $M \neq L$. If $i_0(\phi_M) \geq [\frac{\mathrm{dim}\;\phi_1}{p}]$, then everything follows immediately from \eqref{eq1}, since $i_0(\phi_{L(\phi)}) \geq i_0(\phi_{M(\phi)})$. We may therefore assume that $i_0(\phi_M) < [\frac{\mathrm{dim}\;\phi_1}{p}]$, and hence that $i_0(\phi_{M(\phi)}) - i_1(\phi) \geq i_0(\phi_M)$. Let $\tau = (\phi_M)_{an}$. Then, using Remark \ref{ffremarks} (2), we have
\begin{equation*}  i_0(\phi_{M(\phi)}) = i_0(\phi_M) + i_0(\tau_{M(\tau)}) = i_0(\phi_M) + i_1(\tau). \end{equation*}
In particular, we see that
\begin{equation} \label{eq3} i_1(\tau) \geq i_1(\phi). \end{equation}
Now, the induction hypothesis implies that
\begin{equation} \label{eq4} i_0(\tau_{L(\tau)}) - i_1(\tau) \geq \mathrm{min}\lbrace i_0(\tau_L), [\frac{\mathrm{dim}\;\tau_1}{p}] \rbrace. \end{equation}
By the definition of $\tau$ (and again, Remark \ref{ffremarks} (2)) we have
\begin{align} \label{eq5} i_0(\phi_{L(\phi)}) - i_1(\phi) &= i_0(\phi_M) + i_0(\tau_{L(\tau)}) - i_1(\phi) \nonumber \\
&= i_0(\phi_M) + (i_0(\tau_{L(\tau)}) - i_1(\tau)) + (i_1(\tau) - i_1(\phi)). \end{align}
The proof now splits into two cases.\\

\noindent {\it Case 1.} $i_0(\tau_L) \leq [\frac{\mathrm{dim}\;\tau_1}{p}]$. In this case \eqref{eq3}, \eqref{eq4} and \eqref{eq5} together imply that
\begin{equation*} i_0(\phi_{L(\phi)}) - i_1(\phi) \geq i_0(\phi_M) + i_0(\tau_L). \end{equation*}
But $i_0(\phi_M) + i_0(\tau_L) = i_0(\phi_L)$ by the definition of $\tau$, so we are done with this case.\\

\noindent {\it Case 2.} $i_0(\tau_L) > [\frac{\mathrm{dim}\;\tau_1}{p}]$. Suppose first that $i_0(\phi_M) = 0$, so that $\tau = \phi_M$. In particular, $\mathrm{dim}\;\tau = \mathrm{dim}\;\phi$. Then \eqref{eq3}, \eqref{eq4} and \eqref{eq5} together imply that
\begin{align*} i_0(\phi_{L(\phi)}) - i_1(\phi) &\geq [\frac{\mathrm{dim}\;\tau_1}{p}] + (i_1(\tau) - i_1(\phi))\\
&= [\frac{\mathrm{dim}\;\phi - i_1(\tau)}{p}] + (i_1(\tau) - i_1(\phi))\\
&=  [\frac{\mathrm{dim}\;\phi + (p-1)i_1(\tau) - pi_1(\phi)}{p}]\\
&\geq  [\frac{\mathrm{dim}\;\phi - i_1(\phi)}{p}] = [\frac{\mathrm{dim}\;\phi_1}{p}], \end{align*}
and we are done. We may therefore assume that $i_0(\phi_M) > 0$. Then \eqref{eq3}, \eqref{eq4} and \eqref{eq5} together imply that
\begin{align*} i_0(\phi_{L(\phi)}) - i_1(\phi) & \geq i_0(\phi_M) + [\frac{\mathrm{dim}\;\tau_1}{p}] + (i_1(\tau) - i_1(\phi))\\
& \geq \frac{pi_0(\phi_M) + \mathrm{dim}\;\tau_1 - (p-1) + pi_1(\tau) - pi_1(\phi)}{p}\\
& = \frac{pi_0(\phi_M) + \mathrm{dim}\;\tau - (p-1) + (p-1)i_1(\tau) - pi_1(\phi)}{p}\\
& = \frac{\mathrm{dim}\;\phi_1 + (p-1)(i_0(\phi_M) - 1 + i_1(\tau) - i_1(\phi))}{p}\\
& \geq \frac{\mathrm{dim}\;\phi_1}{p} \geq [\frac{\mathrm{dim}\;\phi_1}{p}], \end{align*}
and again we are done. The theorem is proved. \end{proof} \end{theorem}

We consider some special cases of this result.

\begin{corollary} \label{i1comparison} Let $\phi$ be an anisotropic quasilinear $p$-form of dimension $>1$ over $F$, and let $L$ be a field extension of $F$. If $i_0(\phi_L) \leq [\frac{\mathrm{dim}\;\phi_1}{p}]$, then $i_1((\phi_L)_{an}) \geq i_1(\phi)$.
\begin{proof} Using Remark \ref{ffremarks} (2), we see that $i_1(\phi_{L(\phi)}) = i_0(\phi_L) + i_1((\phi_L)_{an})$. The assertion therefore follows immediately from Theorem \ref{secondcomparison}. \end{proof} \end{corollary}

\begin{corollary} \label{i1i2comparison} Let $\phi$ be an anisotropic quasilinear $p$-form of dimension $>1$ over $F$. Then
\begin{equation*} i_2(\phi) \geq \mathrm{min} \lbrace i_1(\phi), [\frac{\mathrm{dim}\;\phi_1}{p}] \rbrace. \end{equation*}
\begin{proof} Applying Theorem \ref{secondcomparison} in the case where $L = F(\phi)$, we get 
\begin{equation*} i_0(\phi_{F(\phi^{\times 2})}) - i_1(\phi) \geq \mathrm{min} \lbrace i_1(\phi), [\frac{\mathrm{dim}\;\phi_1}{p}] \rbrace. \end{equation*}
It only remains to observe that the left hand side of this inequality is equal to $i_2(\phi)$. This follows from Remark \ref{sspremarks} (2) and Lemma \ref{placeisotropy}. \end{proof} \end{corollary}

These results are of particular interest in the case where $p=2$. We shall explore this further in the remaining sections.

\section{Functoriality of the standard splitting pattern} 

Let $\phi$ and $\psi$ be anisotropic quasilinear $p$-forms of dimension $>1$ over $F$. We have already seen in Proposition \ref{heightfunctoriality} that if $\phi_{F(\psi)}$ is isotropic, then $h(\psi) \leq h(\phi)$. In other words, the existence of a rational map $X_\psi \dashrightarrow X_\phi$ implies the inequality $h(\psi) \leq h(\phi)$. On the other hand, Theorem \ref{compressibilitytheorem} and Proposition \ref{equivalenceandssp} suggest that this ``functoriality'' is exhibited not only by the height, but the standard splitting pattern as a whole. To make this precise, let us introduce the following notation. For $\phi$ and $\psi$ as above, we write $\widetilde{\mathrm{sp}}(\psi) \leq \widetilde{\mathrm{sp}}(\phi)$ if $h(\psi) \leq h(\phi)$ and $\mathrm{dim}\;\psi_r \leq \mathrm{dim}\;\phi_r$ for every $r \in [1,h(\psi)]$. In this section we will prove the following result.

\begin{theorem} \label{functorialityofssp} Assume that $p=2$ or $p=3$. Let $\phi$ and $\psi$ be anisotropic quasilinear $p$-forms of dimension $>1$ over $F$ such that $\phi_{F(\psi)}$ is isotropic. Then
\begin{enumerate} \item[$\mathrm{(1)}$] $\widetilde{\mathrm{sp}}(\psi) \leq \widetilde{\mathrm{sp}}(\phi)$.
\item[$\mathrm{(2)}$] $\widetilde{\mathrm{sp}}(\psi) = \widetilde{\mathrm{sp}}(\phi)$ if and only if $\psi_{F(\phi)}$ is isotropic. \end{enumerate} \end{theorem}

\begin{remark} This result is a substantial generalisation of Theorem \ref{compressibilitytheorem}. As the reader will observe, there are several obstructions which currently limit our approach to the primes 2 and 3. Most of these are related to Question \ref{questions} (1), but we remark that even if we replace the condition ``$\phi_{F(\psi)}$ is isotropic'' with the condition ``there exists an $F$-place $F(\phi) \rightharpoonup F(\psi)$'', our proof of Theorem \ref{functorialityofssp} does not go through for $p>3$. \end{remark}

We will need some preliminary results. We begin with following special case of Theorem \ref{functorialityofssp}, which was proved in \cite{Scully}. In the case where $p=2$, it is originally due to B. Totaro (cf. \citep[Theorem 5.2]{Totaro}).

\begin{lemma}[{cf. \citep[Theorem 5.15]{Scully}}] \label{firststepfunctoriality} Assume that $p=2$ or $p=3$. Let $\phi$ and $\psi$ be anisotropic quasilinear $p$-forms of dimension $>1$ over $F$ such that $\phi_{F(\psi)}$ is isotropic. Then
\begin{enumerate} \item[$\mathrm{(1)}$] $\mathrm{dim}\;\psi_1 \leq \mathrm{dim}\;\phi_1$.
\item[$\mathrm{(2)}$] $\mathrm{dim}\;\psi_1 = \mathrm{dim}\;\phi_1$ if and only if $\psi_{F(\phi)}$ is isotropic. \end{enumerate}
\begin{proof} Let $\sigma$ be a minimal neighbour of $\phi$. By Theorem \ref{ruledness}, we have $F(\sigma) \sim_F F(\phi)$. On the other hand, Proposition \ref{23transitivity} shows that there exists an $F$-place $F(\phi) \rightharpoonup F(\psi)$. Since $F$-places can be composed, there exists an $F$-place $F(\sigma) \rightharpoonup F(\psi)$. By Lemma \ref{placeisotropy}, it follows that $\sigma_{F(\psi)}$ is isotropic. The first part of Theorem \ref{compressibilitytheorem} now implies that
\begin{equation*} \mathrm{dim}\;\psi_1 \leq \mathrm{dim}\;\sigma - 1 = \mathrm{dim}\;\phi_1, \end{equation*}
which proves (1). For (2), if $\mathrm{dim}\;\psi_1 = \mathrm{dim}\;\phi_1$, then the second part of Theorem \ref{compressibilitytheorem} shows that $\psi_{F(\sigma)}$ is isotropic. Using the equivalence $F(\sigma)\sim_F F(\phi)$ and Lemma \ref{placeisotropy}, we see that $\psi_{F(\phi)}$ is isotropic. Conversely, if $\psi_{F(\phi)}$ is isotropic, then both $\mathrm{dim}\;\psi_1 \leq \mathrm{dim}\;\phi_1$ and $\mathrm{dim}\;\phi_1 \leq \mathrm{dim}\;\psi_1$ by (1), whence $\mathrm{dim}\;\psi_1 = \mathrm{dim}\;\phi_1$. \end{proof} \end{lemma}

\begin{proposition} \label{towerplaces} Assume that $p=2$ or $p=3$. Let $\phi$ and $\psi$ be anisotropic quasilinear $p$-forms of dimension $>1$ over $F$ such that $\phi_{F(\psi)}$ is isotropic. Let $h = h(\psi)$. Then $h \leq h(\phi)$ and
\begin{enumerate} \item[$\mathrm{(1)}$] There exist $F$-places $F(\psi^{\times (r-1)} \times \phi) \rightharpoonup F(\psi^{\times r})$ for all $r \in [1,h)$.
\item[$\mathrm{(2)}$] There exist $F$-places $F(\phi^{\times r}) \rightharpoonup F(\psi \times \phi^{\times (r-1)})$ for all $r \in [1,h)$. \end{enumerate}
\begin{remark} For $r>1$, these statements do not seem to be immediate consequences of Proposition \ref{23transitivity}. \end{remark}
\begin{proof} The inequality $h \leq h(\phi)$ follows from Proposition \ref{heightfunctoriality}. We will just prove (1). The proof of (2) proceeds along similar lines and we leave it to the reader. We proceed by induction on $r$. The case where $r = 1$ follows from Proposition \ref{23transitivity}. Suppose now that $r>1$ (in particular, we have $h>2$), and let $L = F(\psi^{\times (r-2)})$. Let $\sigma = (\phi_L)_{an}$ and $\tau = (\psi_L)_{an}$. Since $h(\phi) \geq h$, $\sigma$ is not completely split by Lemmas \ref{ffisotropy} and \ref{isotropytower} (3). By Remark \ref{ffremarks} (2), we have equivalences $L(\phi) \sim_L L(\sigma)$ and $L(\psi) \sim_L L(\tau)$. By the induction hypothesis, there exists an $F$-place $L(\sigma) \rightharpoonup L(\psi)$, and hence an $F$-place $L(\phi) \rightharpoonup L(\tau)$. Applying Lemma \ref{placeisotropy} to the form $\phi$, we get
\begin{equation*} i_0(\sigma_{L(\tau)}) = i_0(\phi_{L(\tau)}) - i_0(\phi_L) \geq i_0(\phi_{L(\sigma)}) - i_0(\phi_L) =  i_1(\sigma) > 0. \end{equation*}
In other words, $\sigma$ becomes isotropic over $L(\tau)$. On the other hand, we have equivalences
\begin{equation*} F(\psi^{\times (r-1)} \times \phi) \sim_F L(\tau \times \sigma)\;\;\text{and}\;\;F(\psi^{\times r}) \sim_F F(\tau^{\times 2}) \end{equation*}
by Remarks \ref{ffremarks} (2) and \ref{sspremarks} (2). Replacing $F$ by $L$, and $\phi$ and $\psi$ by $\sigma$ and $\tau$ respectively, we are reduced to the case where $r=2$. Let $\eta = (\phi_{F(\psi)})_{an}$. Then $F(\psi)(\eta) \sim_F F(\psi \times \phi)$ by Remark \ref{ffremarks} (2). It then follows from Proposition \ref{23transitivity} that in order to prove the existence of an $F$-place $F(\psi \times \phi) \rightharpoonup F(\psi^{\times 2})$, it suffices to show that $\eta_{F(\psi^{\times 2})}$ is isotropic. But by Proposition \ref{secondcomparisonbasic}, we have
\begin{equation*} i_0(\eta_{F(\psi^{\times 2})}) = i_0(\phi_{F(\psi^{\times 2})}) - i_0(\phi_{F(\psi)}) \geq \mathrm{min}\lbrace i_0(\phi_{F(\psi)}),[\frac{\mathrm{dim}\;\psi_1}{p}] \rbrace. \end{equation*}
Since $i_0(\phi_{F(\psi)}) > 0$ by assumption, the integer on the right of this inequality is positive provided that $\mathrm{dim}\;\psi_1 \geq p$. Since $p \leq 3$, the condition is automatically satisfied if $h > 2$, and so the statement is proved. \end{proof} \end{proposition}

The following result generalises Lemma \ref{ffsubformtower} in the case where $p \leq 3$.

\begin{proposition} \label{functorialityprop} Assume that $p=2$ or $p=3$. Let $\phi$ and $\psi$ be anisotropic quasilinear $p$-forms of dimension $>1$ over $F$ such that $\phi_{F(\psi)}$ is isotropic. Assume that $\psi_{F(\phi)}$ is anisotropic, and let $(F_r)$ be the standard splitting tower of $\psi$. Then either
\begin{enumerate} \item[$\mathrm{(1)}$] $h(\psi_{F(\phi)}) = h(\psi)$ and $\mathrm{sp}(\psi_{F(\phi)}) = \mathrm{sp}(\psi)$, or
\item[$\mathrm{(2)}$] $h(\psi_{F(\phi)}) = h(\psi) - 1$ and \begin{equation*} \mathrm{sp}(\psi_{F(\phi)}) = (\mathrm{dim}\;\psi, \mathrm{dim}\;\psi_1,...,\mathrm{dim}\;\psi_{s-1},\mathrm{dim}\;\psi_{s+1},...,\mathrm{dim}\;\psi_{h(\psi)} = 1), \end{equation*}
where $s \in [1,h(\psi))$ is the least positive integer $r$ such that $\psi_r$ becomes isotropic over $F_r(\phi)$. \end{enumerate}
\begin{proof} Recall from the proof of Lemma \ref{heightext} that the tower $(F_r(\phi))$ computes the standard splitting pattern of $\psi_{F(\phi)}$. If $h(\psi_{F(\phi)}) = h(\psi)$, then $\mathrm{sp}(\psi_{F(\phi)}) = \mathrm{sp}(\psi)$ by Lemma \ref{heightext}. If $h(\psi_{F(\phi)}) < h(\psi)$, then the same result shows that some $\psi_r$ becomes isotropic over $F_r(\phi)$. Choose $s$ to be minimal among all $r$ with this property. Then $\mathrm{dim}\;(\psi_{F_r(\phi)})_{an} = \mathrm{dim}\;\psi_r$ for all $r < s$. To finish the proof, we need to show that $\mathrm{dim}\;(\psi_{F_r(\phi)})_{an} = \mathrm{dim}\;\psi_{r+1}$ for all $r \in [s,h(\psi) - 1]$. Since $h(\psi_{F(\phi)}) < h(\psi)$, both $\psi_{F_{h(\psi)-1}(\phi)}$ and $\psi_{h(\psi)}$ are completely split, so the statement is clear if $r = h(\psi) - 1$. For the remaining cases, note that we have $F$-places $F_r(\phi) \rightharpoonup F_{r+1}$ for all $r \in [s,h(\psi) - 2]$ by Proposition \ref{towerplaces} (1) and Remark \ref{sspremarks} (2). By Lemma \ref{placeisotropy}, it follows that
\begin{equation*} \mathrm{dim}\;(\psi_{F_r(\phi)})_{an} \geq \mathrm{dim}\;(\psi_{F_{r+1}})_{an} = \mathrm{dim}\;\psi_{r+1} \end{equation*}
for all $r \in [s,h(\psi) - 2]$. It remains to establish the reverse inequalities. First note that if $\psi_r$ becomes isotropic over $F_r(\phi)$ for any such $r$, then the desired inequality
\begin{equation*} \mathrm{dim}\;(\psi_{F_r(\phi)})_{an} \leq \mathrm{dim}\;\psi_{r+1} \end{equation*}
holds by Corollary \ref{23minimalityi1} (that is, by the ``minimality'' property of the index $i_1$). Since $\psi_s$ becomes isotropic over $F_s(\phi)$, this holds in particular when $r = s$. To complete the proof, it will therefore be enough to show that if $\psi_r$ becomes isotropic over $F_r(\phi)$ for any $r \in [s,h(\psi) - 3]$, then $\psi_{r+1}$ becomes isotropic over $F_{r+1}(\phi)$. Fix such an $r$, and let $\eta = (\phi_{F_r})_{an}$. Then $F_r(\eta) \sim_{F_r} F_r(\phi)$ by Remark \ref{ffremarks} (2), and hence $\psi_r$ becomes isotropic over $F_r(\eta)$ (using Lemma \ref{placeisotropy}). Now, since $\phi_{F(\psi)}$ is isotropic, we have $h(\psi) \leq h(\phi)$ by Proposition \ref{heightfunctoriality}. On the other hand, Lemmas \ref{ffisotropy} and \ref{isotropytower} (3) show that $h(\eta) \geq h(\phi) - r$, and so $h(\eta) \geq h(\psi) - r \geq 3$. By Proposition \ref{towerplaces} (2), we have an $F_r$-place $F_r(\psi_r^{\times 2}) \rightharpoonup F_r(\eta \times \psi_r)$. But $F_r(\psi_r^{\times 2}) \sim_{F_{r+1}} F_{r+2}$ and $F_r(\eta \times \psi_r) \sim_{F_{r+1}} F_{r+1}(\phi)$ by Remark \ref{sspremarks} (2). Hence we have an $F$-place $F_{r+2} \rightharpoonup F_{r+1}(\phi)$. In particular, we have
\begin{equation*} i_0((\psi_{r+1})_{F_{r+1}(\phi)}) = i_0(\psi_{F_{r+1}(\phi)}) - i_0(\psi_{F_{r+1}}) \geq i_0(\psi_{F_{r+2}}) - i_0(\psi_{F_{r+1}}) = i_1(\psi_{r+1}) \geq 1 \end{equation*}
by Lemma \ref{placeisotropy}. In other words, $\psi_{r+1}$ becomes isotropic over $F_{r+1}(\phi)$, as we wanted. \end{proof} \end{proposition}

Now we can complete the proof of Theorem \ref{functorialityofssp}.

\begin{proof}[Proof of Theorem \ref{functorialityofssp}.] If $\psi_{F(\phi)}$ is also isotropic, then we have $F(\psi) \sim_F F(\phi)$ by Proposition \ref{23transitivity}. By Proposition \ref{equivalenceandssp}, we conclude that $\widetilde{\mathrm{sp}}(\psi) = \widetilde{\mathrm{sp}}(\phi)$. Conversely, if $\widetilde{\mathrm{sp}}(\psi) = \widetilde{\mathrm{sp}}(\phi)$, then in particular we have $\mathrm{dim}\;\psi_1 = \mathrm{dim}\;\phi_1$, and hence $\psi_{F(\phi)}$ is isotropic by the second part of Lemma \ref{firststepfunctoriality}. This proves (2).

To prove (1), we may assume that $\psi_{F(\phi)}$ is anisotropic. We argue by induction on $h(\phi)$. By Lemma \ref{firststepfunctoriality} (1), we have $\mathrm{dim}\;\psi_1 \leq \mathrm{dim}\;\phi_1$. In particular, if $h(\psi) \leq 2$, then we are done. We may therefore assume that $h(\psi) > 2$, and we need to show that $\widetilde{\mathrm{sp}}(\psi_1) \leq \widetilde{\mathrm{sp}}(\phi_1)$. Since $h(\psi) >2$, we have an $F$-place $F(\phi^{\times 2}) \rightarrow F(\psi \times \phi)$ by part (2) of Proposition \ref{towerplaces}. By Lemma \ref{placeisotropy} and Remark \ref{sspremarks} (2), we have
\begin{equation*} i_0((\phi_1)_{F(\psi \times \phi)}) = i_0(\phi_{F(\psi \times \phi)}) - i_1(\phi) \geq i_0(\phi_{F(\phi^{\times 2})}) - i_1(\phi) = i_1(\phi_1) = i_2(\phi) \geq 1. \end{equation*}
Hence $\phi_1$ becomes isotropic over $F(\psi \times \phi) = F(\phi)(\psi_{F(\phi)})$. By the induction hypothesis, we have
\begin{equation*} \widetilde{\mathrm{sp}}(\psi_{F(\phi)}) \leq \widetilde{\mathrm{sp}}(\phi_1). \end{equation*}
It will therefore be enough to show that $\widetilde{\mathrm{sp}}(\psi_1) \leq \widetilde{\mathrm{sp}}(\psi_{F(\phi)})$. This follows immediately from Proposition \ref{functorialityprop}, and the theorem is proved. \end{proof}

\section{Applications to quasilinear quadratic forms}

In this final section, we apply the results of the previous sections in the special case where $p=2$. It will be convenient to make the following remark.

\begin{remark} \label{p=2remark} Assume that $p=2$. Then in the statements of Theorem \ref{secondcomparison} and Corollaries \ref{i1comparison} and \ref{i1i2comparison}, the integer $[\frac{\mathrm{dim}\;\phi_1}{2}]$ can be replaced by $[\frac{\mathrm{dim}_{Izh}\phi}{2}]$. The reader will easily observe that it suffices to verify that the integer $[\frac{\mathrm{dim}\;\psi_1}{2}]$ appearing in the statement of Proposition \ref{secondcomparisonbasic} can be replaced by $[\frac{\mathrm{dim}_{Izh}\psi}{2}]$. To see this, let $s' = \mathrm{min}\lbrace i_0(\phi_{F(\sigma)}), [\frac{\mathrm{dim}_{Izh}\psi}{2}] \rbrace$. As in the proof of Proposition \ref{secondcomparisonbasic}, we find $\tau' \subset \phi$ such that $\mathrm{dim}\;\tau' \leq 2s$ and $i_0(\tau'_{F(\sigma)}) \geq s$. However, using Corollary \ref{isotropyquadratic}, we see that the form $\tau'$ is in this case divisible by the form $\pfister{a}$. One easily checks that this divisibility implies that $i_1(\tau') > 1$ (see \citep[Proposition 4.19]{Hoffmann2} for example). In particular, we have $\mathrm{dim}\;\tau'_1 \leq 2s' - 2 \leq \mathrm{dim}_{Izh}\psi - 2 <\mathrm{dim}\;\psi_1$. By Lemma \ref{firststepfunctoriality}, this implies that $\tau'$ remains anisotropic over $F(\psi)$. The reader will easily check that after replacing $s$ with $s'$ and $\tau$ with $\tau'$, one can argue exactly as before to obtain the improved result. \end{remark}

From the remainder of this paper, we assume that $p=2$. We begin by proving an analogue of Vishik's Theorem \ref{outerexcellentconnections} for quasilinear quadratic forms.

\begin{theorem} \label{excellentconnectionsquasilinear} Let $\phi$ be an anisotropic quasilinear quadratic form of dimension $>1$ over $F$, and write $\mathrm{dim}\;\phi = 2^n + m$ for uniquely determined integers $n \geq 0$ and $m \in [1,2^n]$. Let $L$ be a field extension of $F$. If $i_0(\phi_L) < m$, then $i_0(\phi_L) \leq m - i_1(\phi)$.
\begin{proof} By Theorem \ref{secondcomparison} and Remark \ref{p=2remark}, we have
\begin{equation} \label{eq6} i_0(\phi_{L(\phi)}) - i_1(\phi) \geq \mathrm{min}\lbrace i_0(\phi_L),[\frac{\mathrm{dim}_{Izh}\phi}{2}]\rbrace. \end{equation}
Let $\tau = (\phi_L)_{an}$, so that $\mathrm{dim}\;\tau = 2^n + (m-i_0(\phi_L))$. Since $i_0(\phi_L) < m$, Corollary \ref{Hoffmannbound} implies that 
\begin{equation} \label{eq7} i_1(\tau) \leq m - i_0(\phi_L). \end{equation} 
Now, using Remark \ref{ffremarks} (2), we see that $i_0(\phi_{L(\phi)}) = i_0(\phi_L) + i_1(\tau)$. Equation \eqref{eq7} therefore implies that $i_0(\phi_{L(\phi)}) \leq m$, and \eqref{eq6} then gives
\begin{equation*} m - i_1(\phi) \geq \mathrm{min}\lbrace i_0(\phi_L),[\frac{\mathrm{dim}_{Izh}\phi}{2}]\rbrace. \end{equation*}
To complete the proof, it is now enough to show that $[\frac{\mathrm{dim}_{Izh}\phi}{2}] > m-i_1(\phi)$. This is an easy calculation which follows directly from the definitions, and we leave it to the reader. \end{proof} \end{theorem}

\begin{remark} Let $E$ be a field of characteristic $p>0$, and let $\phi$ be an anisotropic quasilinear $p$-form over $E$ with $\mathrm{dim}\;\phi = p^n + m$ for integers $n \geq 0$ and $m \in [1,p^{n+1} - p^n]$. Then the inequality $[\frac{\mathrm{dim}_{Izh}\phi}{p}] > m-i_1(\phi)$ does not hold in general if $p>2$. \end{remark}

As a corollary of Theorem \ref{excellentconnectionsquasilinear}, we get the following result, which significantly improves Hoffmann and Laghribi's Corollary \ref{Hoffmannbound}. 

\begin{theorem} \label{i1bounds} Let $\phi$ be an anisotropic quasilinear quadratic form of dimension $>1$ over $F$, and write $\mathrm{dim}\;\phi = 2^n + m$ for uniquely determined integers $n \geq 0$ and $m \in [1,2^n]$. Then either $i_1(\phi) = m$ or $i_1(\phi) \leq \frac{m}{2}$.
\begin{proof} Suppose that $i_1(\phi) \neq m$. By Corollary \ref{Hoffmannbound}, we therefore have $i_1(\phi) < m$. Applying Theorem \ref{excellentconnectionsquasilinear} in the case $L = F(\phi)$, we see that $i_1(\phi) \leq m - i_1(\phi)$, and the result follows. \end{proof} \end{theorem}

Recall that given an anisotropic quasilinear quadratic form $\phi$ of dimension $>1$ over $F$, $h_{qp}(\phi)$ denotes the smallest nonnegative integer $r$ for which $\phi_r$ is a quasi-Pfister neighbour. By Theorem \ref{pfisterneighbourclassification}, the form $\psi_{h_{qp}(\phi)+1}$ is similar to a quasi-Pfister form, and so the standard splitting of $\phi$ is completely understood from this point on (cf. Theorem \ref{ssppfister}). We can now say something about the structure of the ``nontrivial'' part of the standard splitting pattern.

\begin{theorem} \label{hilltheorem} Let $\phi$ be an anisotropic quasilinear quadratic form of dimension $>1$ over $F$. Assume that $\phi$ is not a quasi-Pfister neighbour (that is, $h_{qp}(\phi) >0$). Then $i_1(\phi) \leq i_2(\phi) \leq ... \leq ... i_{h_{qp}}(\phi)$.
\begin{proof} It is sufficient to prove that $i_1(\phi) \leq i_2(\phi)$. By Corollary \ref{i1i2comparison} and Remark \ref{p=2remark}, we have
\begin{equation*} i_2(\phi) \geq \mathrm{min}\lbrace i_1(\phi),[\frac{\mathrm{dim}_{Izh}\phi}{2}] \rbrace \end{equation*}
It will be enough to show that the right hand side is equal to $i_1(\phi)$. Suppose not. Then
\begin{equation*} i_1(\phi_1) = i_2(\phi) \geq [\frac{\mathrm{dim}_{Izh}\phi}{2}] = [\frac{\mathrm{dim}\;\phi_1 + 1}{2}] \end{equation*}
It follows from Lemma \ref{isotropytower} (2), that we must in fact have $i_1(\phi_1) = \frac{\mathrm{dim}\;\phi_1}{2}$. By Theorem \ref{ssppfister}, $\phi_1$ is therefore similar to a quasi-Pfister form. But by Theorem \ref{pfisterneighbourclassification}, this implies that $\phi$ is a quasi-Pfister neighbour, contradicting our assumption. The result follows. \end{proof} \end{theorem}

As a final application of our results, we give a positive answer to Question \ref{maxsplittingquestions} (1).

\begin{theorem} \label{quadraticmaxsplitting} Let $\phi$ be an anisotropic quasilinear quadratic form over $F$ such that $2^n + 2^{n-2} < \mathrm{dim}\;\phi \leq 2^{n+1}$ for some $n \geq 2$. If $\phi$ has maximal splitting, then $\phi$ is a quasi-Pfister neighbour. In other words, the analogue of Conjecture \ref{maxsplittingconjecture} for quasilinear quadratic forms is true.
\begin{proof} Suppose that $\phi$ is not a quasi-Pfister neighbour. Then $i_2(\phi) \geq i_1(\phi)$ by Theorem \ref{hilltheorem}. In particular, we have
\begin{align*} i_1(\phi_1) = i_2(\phi) & \geq i_1(\phi)\\
&> 2^{n-2}\\
&= \frac{2^{n-1}}{2}\\
&= \frac{2^n - 2^{n-1}}{2} \\
&= \frac{\mathrm{dim}\;\phi_1 - 2^{n-1}}{2}. \end{align*}
By Theorem \ref{i1bounds}, we see that we must have $i_1(\phi_1) = 2^{n-1} = \frac{\mathrm{dim}\;\phi_1}{2}$, and hence $\phi_1$ is similar to a quasi-Pfister form by Theorem \ref{ssppfister}. But by Theorem \ref{pfisterneighbourclassification}, this implies that $\phi$ is a quasi-Pfister neighbour, contradicting our assumption. \end{proof} \end{theorem}

\begin{remarks} \label{concludingremarks} We conclude with some general remarks on the results of this section.
\begin{enumerate} \item[$\mathrm{(1)}$] Theorem \ref{i1bounds} represents an important step towards an analogue of Karpenko's Theorem \ref{Hoffmann'sconjecture} for quasilinear quadratic forms. In order to establish such an analogue in full generality, Corollary \ref{shiftinglemma} and Theorem \ref{pfisterneighbourclassification} show that it is sufficient to consider forms $\phi$ satisfying $h_{qp}(\phi) = 1$. The invariant $h_{qp}$ may be seen to provide a certain measure of ``complexity'', with quasi-Pfister neighbours being the forms of ``lowest complexity'' in this system. It is not clear how useful this invariant is in classifying forms of ``higher complexity''. In view of the above remarks, it would already be very interesting to know if one can say something about those forms which belong to the second level of this classification (that is, those forms $\phi$ with $h_{qp}(\phi) = 1$).
\item[$\mathrm{(2)}$] Theorem \ref{i1bounds} is itself a consequence of Theorem \ref{excellentconnectionsquasilinear}, and the latter result is a direct analogue of Theorem \ref{outerexcellentconnections} for quasilinear quadratic forms. As discussed in \S 1, Theorem \ref{outerexcellentconnections} is a formal corollary of A. Vishik's theorem on ``excellent connections'' in the motivic theory of quadrics (\citep[Theorem 1.3]{Vishik2}). The proof, however, only relies on a special case of this deep result, namely the ``outer excellent connections''. As explained in \cite{Vishik2}, the ``inner excellent connections'' established by Vishik also provide nontrivial relations between the higher Witt indices, and these relations put significant restrictions on the possible values of the invariant $i_1$. In fact, they can be used to give another proof of Karpenko's Theorem \ref{Hoffmann'sconjecture} (cf. \citep[Theorem 2.5]{Vishik2}). In view of Theorem \ref{excellentconnectionsquasilinear}, it is an interesting question to ask if analogues of the other relations established in \cite{Vishik2} exist for quasilinear quadratic forms. It seems that the methods used in the present article are insufficient to address this problem beyond Theorem \ref{excellentconnectionsquasilinear}. \end{enumerate} \end{remarks}

\appendix

\section{Places} 

We briefly recall some standard facts about places. All details which are not provided can be found in \S 103 of \cite{EKM}. Throughout this section we fix an arbitrary base field $k$.

\begin{definition} Let $K$ and $L$ be field extensions of $k$. A \emph{k-place} $K \rightharpoonup L$ consists of the following data.
\begin{itemize} \item A valuation subring $R$ of $K$ containing $k$.
\item A local $k$-algebra homomorphism $R \rightarrow L$. \end{itemize} \end{definition}

If $K \rightharpoonup L$ and $L \rightharpoonup M$ are $k$-places defined by local $k$-algebra homomorphisms $f \colon R \rightarrow L$ and $g \colon S \rightarrow M$ respectively, then $T = f^{-1}(S)$ is a valuation subring of $K$ containing $k$ and the restriction $g \circ f|_T \colon T \rightarrow M$ is a local $k$-algebra homomorphism. In this way, the composition of $k$-places is defined.\\

Our interest in places is primarily motivated by the following lemma, which follows directly from the valuative criterion of properness.

\begin{lemma} \label{placescomplete} Let $K$ and $L$ be field extensions of $k$, and let $X$ be a complete variety over $k$. Assume that there exists a $k$-place $K \rightharpoonup L$. If $X(K) \neq \emptyset$, then $X(L) \neq \emptyset$. \end{lemma}

\begin{example} \label{placecompletefunctionfield} Let $L$ be a field extension of $k$, and let $X$ be a complete variety over $k$. If there exists a $k$-place $k(X) \rightharpoonup L$, then $X(L) \neq \emptyset$. \end{example}

Let $k \subset K \subset L$ be a tower of fields over $k$. Then $K$ is a valuation subring of itself, and the inclusion $K \subset L$ defines a $k$-place $K \rightharpoonup L$, called the \emph{trivial place}. More subtle examples of places are given by the following result.

\begin{lemma} \label{placeregpoint} Let $X$ be a variety over $k$, and let $x \in X$ be a regular point. Then there exists a $k$-place $k(X) \rightharpoonup k(x)$. \end{lemma}

\begin{example} \label{placepurelytran} Let $k \subset L$ be an extension of fields, and let $K$ be a purely transcendental extension of $L$ (of finite transcendence degree). Then there exists a $k$-place $K \rightharpoonup L$. This follows by applying Lemma \ref{placeregpoint} to a regular model of $K$ over $L$ possessing a rational point (for example, affine $L$-space of the appropriate dimension). \end{example}

We conclude this section by introducing the following definition.

\begin{definition} \label{placeequivalence} Let $K$ and $L$ be field extensions of $k$. We say that $K$ and $L$ are \emph{$k$-equivalent}, and write $K \sim_k L$, if there exist $k$-places $K \rightharpoonup L$ and $L \rightharpoonup K$. \end{definition}

\begin{example} \label{stabbireq} Let $X$ and $Y$ be stably birational varieties over $k$. Then $k(X) \sim_k k(Y)$. Indeed, there exists a field $K$ which is a purely transcendental extension (of finite transcendence degree) of both $k(X)$ and $k(Y)$. By Example \ref{placepurelytran}, there exist $k$-places $E \rightharpoonup k(Y)$ and $E \rightharpoonup k(X)$. Composing these with the trivial places $k(X) \rightharpoonup E$ and $k(Y) \rightharpoonup E$ respectively, we get $k$-places $k(X) \rightharpoonup k(Y)$ and $k(Y) \rightharpoonup k(X)$. \end{example}

\bibliographystyle{alphaurl}
\bibliography{SQPF}

\end{document}